\documentclass[a4paper,10pt]{article}

\usepackage{fouriernc}

\usepackage{amsmath,amsthm,eufrak,latexsym}

\usepackage[margin=30mm]{geometry}

\newtheorem{theorem}{Theorem}[section]

\newtheorem{lemma}[theorem]{Lemma}
\newtheorem{corollary}[theorem]{Corollary}
\newtheorem{remark}[theorem]{Remark}
\newtheorem{definition}[theorem]{Definition}
\newtheorem{example}[theorem]{Example}
\newtheorem{assumption}[theorem]{Assumption}

\usepackage{xcolor}

\def\cX{\mathcal{X}}
\def\mR{\mathbb{R}}
\def\mZ{\mathbb{Z}}
\def\mN{\mathbb{N}}
\def\eps{\varepsilon}

\begin{document}

\title{Invariant measures for multidimensional\\ fractional stochastic volatility models
\thanks{Both authors enjoyed the support of 
the ``Lend\"ulet'' grant LP 2015-6 of the
Hungarian Academy of Sciences. The first author was also supported by NRDI (National
Research, Development and Innovation Office) grant FK 135711, the
J\'anos Bolyai Research Scholarship of the Hungarian Academy of Sciences and by
the ÚNKP-20-5 New National Excellence Program of the Ministry for Innovation and Technology from
the source of the National Research, Development and Innovation Fund.}}

\author{Bal\'azs Gerencs\'er\thanks{Alfr\'ed R\'enyi Institute of Mathematics and 
E\"otv\"os Lor\'and University, Budapest, Hungary}
\and Mikl\'os R\'asonyi\thanks{Alfr\'ed R\'enyi Institute of Mathematics, Budapest, Hungary; rasonyi@renyi.hu}}

\date{\today}
 
\maketitle
{}

\centerline{Dedicated to Istv\'an Gy\"ongy on the occasion of his 70th birthday}

\begin{abstract} 

We establish convergence to an invariant measure as time tends to infinity,
for a large class of (possibly non-Markovian) stochastic volatility models. Our arguments are based on 
a novel coupling idea for Markov chains which also extends to Markov chains in random environments
in an efficient way. 

\end{abstract}



\section{Introduction}

Stochastic volatility models (in the simplest one-dimensional case) are of the form
\begin{eqnarray}\label{diffi1}
dS_{t} &=& \nu_{1}(S_{t},V_{t})S_{t}\, dt+ V_{t}S_{t}\, d\overline{W}_{t},
\end{eqnarray}	
where $\overline{W}$ is a Brownian motion, $\nu_1$ is a suitable function
and $S$ describes the (discounted) price of an asset with volatility process $V$.

The present paper is about the long-term behaviour of $S$.
In the Markovian case, $V$ satisfies a
stochastic differential equation (SDE), 
\begin{eqnarray}
dV_{t} &=& \nu_{2}(V_{t})\, dt+\sigma(V_{t})\, dB_{t},\label{diffi2}
\end{eqnarray}
where $B$ is another Brownian motion, possibly correlated with $\overline{W}$; 
$\nu_{2},\sigma$ are suitable
functions. In such diffusion models 
there is an arsenal of techniques from Markov process theory
to show that the law of $S_t$ tends to a limit as $t\to\infty$, see 
e.g.\ \cite{veretennikov,mt2,mt3,down,rt,ver2,kha}, Chapter 20 of \cite{mt} and Subsection 7.1
of \cite{hairer}. 

Recently, however, fractional stochastic volatility models have become popular
(see \cite{cr,gjr,website}), where the process $V$ is not Markovian. 
For instance, 
\begin{equation}\label{vfrac}
V_t=\exp\left(J_t\right),\quad J_t:=\int_{-\infty}^t K(t-s)dB_s,
\end{equation}
with some (two-sided) Brownian motion $B$, and a suitable function $K:\mathbb{R}_+\to\mathbb{R}$. 
In such a setting the question of stochastic stability becomes difficult, one cannot rely on
the usual Markovian techniques and there seems to
be no results in the literature that would imply the convergence of the 
law of $S_{t}$ as $t\to\infty$ at this level of generality.

We now explain our motivations for studying such models.
Asset price processes often show mean-reversion 
(for instance, commodities or commodity
futures, see \cite{commodity1,commodity2}).
Optimal investment problems for such assets 
were considered in \cite{timid}, see also the 
study \cite{asymp} on asymptotic arbitrage. Long-term investments may also be studied in 
the framework of ergodic, risk-sensitive or adaptive control (see 
e.g.\ \cite{hernandez-lerma,lukasz,pitera,bmp}).
All these approaches require that the law of $S_{t}$ should converge 
to a steady state as $t\to\infty$.
Long-term investment problems
for fractional processes were treated in \cite{grs,negmem}, but these studies do not cover
fractional stochastic volatility models. 

The present paper proves that --
under mean-reversion and smoothness conditions on the drift of $S$ and integrability
assumptions on $V_{0}$, $S_{0}$ -- the stochastic system $(S_{t},V_{t})$
converges to an invariant probability, independent of the initialization $S_{0}$.  
A multi-asset framework is treated and $B,\overline{W}$ will be allowed to
have a stochastic correlation.
Our arguments are based on a new
coupling construction for (discrete-time) Markov processes in random environments.

In the extant literature on fractional volatility, asset dynamics is most often considered for 
purposes of derivative pricing; \cite{friz,fukasawa,jacquier}
are early examples. These papers thus work under the risk-neutral measure,
which corresponds to taking $\nu_{1}\equiv 0$ in \eqref{diffi1}. 
As we have in mind a different class of problems (portfolio optimization),
we work under the physical probability, where $\nu_{1}$ is non-zero.

In Section \ref{resi} we rigorously formulate our main results, Theorems 
\ref{stability} and \ref{stability1}.
A novel (discrete-time) coupling method is introduced in Section \ref{couplingsec}.
As a warm-up, it is first presented for (ordinary) Markov chains in Subsection \ref{marki}. 
Subsection \ref{nonmarki} develops the same ideas in the more involved setting
of Markov chains in random environments. In Section \ref{stabi} we prove the main results,
combining advanced Malliavin calculus 
techniques with the discrete-time construction of Subsection \ref{nonmarki}. 

\section{Results}\label{resi}

Scalar product in finite-dimensional Euclidean spaces is denoted by $\langle\cdot,\cdot\rangle$,{}
the coresponding norm is $|\cdot|$, where the dimension of the space may vary. For a matrix $A$,
$A^{*}$ denotes its transpose. For 
matrices $A$, $|A|$ denotes the operator norm.

All the random objects in the present paper will live on a
fixed probability space $(\Omega,\mathcal{F},P)$. For a Polish space $\mathcal{Z}$,
its Borel sigma-algebra is denoted $\mathcal{B}(\mathcal{Z})$.
If $Z:\Omega\to\mathcal{Z}$ is $\mathcal{F}/\mathcal{B}(\mathcal{Z})$ is measurable
(that is, if $Z$ is a $\mathcal{Z}$-valued random variable) then
$\mathcal{L}(Z)$ denotes its law on $\mathcal{B}(\mathcal{Z})$. 

Fix $d,m\in\mathbb{N}\setminus\{0\}$ with $d\leq m$. The number of assets will be $d$ and 
the dimension of the driving noise $m$. For every $k\geq 1$, let $\mathcal{W}^k$ 
denote the set of continuous $\mathbb{R}^{k}$-valued 
functions on $\mathbb{R}$ which is a Polish space under the metric
$$
\mathbf{d}_{k}(f,g):=\sum_{i=-\infty}^{\infty}\frac{1}{2^{|i|}}\left[{}
1\wedge \sup_{u\in [i,i+1]}|f(u)-g(u)|\right],\quad f,g\in\mathcal{W}^k.
$$

Let $B_t$, $t\in\mathbb{R}$ be a two-sided $m$-dimensional Brownian motion 
(i.e. $B_t$, $B_{-t}$, $t\in\mathbb{R}_+$ are
independent standard $m$-dimensional Brownian motions), 
let $\mathcal{G}_t$, $t\in\mathbb{R}$ denote its completed natural
filtration. Let $\mathcal{V}$ denote the set of  
$d\times d$ non-singular matrices, $\mathcal{R}$ the
set of $d\times m$ matrices $r$ satisfying  $rr^{*}<I$,
where $I$ is the $d$-dimensional identity matrix and $A<B$ for
symmetric, positive semidefinite $d\times d$ matrices $A,B$ means that
$B-A$ is positive definite. Similarly, $A\leq B$ means that $B-A$
is positive semindefinite and $\sqrt{A}$ denotes the usual square root
of semidefinite matrices.

Let $V_t$, $t\in\mathbb{R}$ (resp.\ $\rho_t$, $t\in\mathbb{R}$) be $\mathcal{V}$-valued (resp.\
$\mathcal{R}$-valued) processes with continuous trajectories.

Notice that $\mathbf{B}_{t}:=(B_{t}-B_{t+s})_{s\in\mathbb{R}}$ (resp.
$\mathbf{V}_{t}:=(V_{t+s})_{s\in\mathbb{R}}$ and $\mathbf{R}_{t}:=(\rho_{t+s})_{s\in\mathbb{R}}$) 
can be naturally regarded as a $\mathcal{W}^m$-valued (resp. 
$\mathcal{W}^{d\times d}$-valued and $\mathcal{W}^{d\times m}$-valued)
random process indexed by $t\in\mathbb{R}$. 

\begin{assumption}\label{stationary}
There are measurable functions $F_{1}:\mathcal{W}^{m}\to\mathcal{W}^{d\times d}$,
$F_{2}:\mathcal{W}^{m}\to\mathcal{W}^{d\times m}$ such that 
$\mathbf{V}_{t}=F_{1}(\mathbf{B}_{t})$,
$\mathbf{R}_{t}=F_{2}(\mathbf{B}_{t})$. Furthermore, $(V_{t},\rho_{t})$, 
$t\in\mathbb{R}$ is adapted
to $\mathcal{G}_{t}$, $t\in\mathbb{R}$.
\end{assumption}

In plain English, $(V_{t},\rho_{t})$ is a nonanticipative functional of 
the increments of the Brownian motion $B$
up to $t$. A specification like \eqref{vfrac} is a typical example.
Under Assumption \ref{stationary}, 
$(\mathbf{V}_{t},\mathbf{R}_{t},\mathbf{B}_{t})$, $t\in\mathbb{R}$ 
is a stationary process in the strong sense.

Let $W_t$, $t\in\mathbb{R}_+$ be another, $d$-dimensional standard Brownian motion with (completed) natural
filtration $\mathcal{F}_t$, $t\in\mathbb{R}_+$.  
Instead of prices, it is more convenient to
work with log-prices. Hence we consider $d$ financial assets
whose log-price is given by the $d$-dimensional process $L_t$, $t\in\mathbb{R}_+$ which 
is the solution of
the stochastic differential equation
\begin{equation}\label{startrek}
dL_t=\zeta(L_t,V_{t})\, dt +  V_t \rho_t \, dB_t+V_t \sqrt{I-\rho_t\rho_{t}^{*}}\, dW_t,
\end{equation}
where $L_0$ is a random variable and 
$\zeta:\mathbb{R}^{d}\times \mathcal{V}\to\mathbb{R}^{d}$ is
a measurable function. 

Assumptions \ref{initial}, \ref{thrice} and \ref{ve}, stipulated below, 
guarantee a unique $(\mathcal{F}_{t}\vee\mathcal{G}_{t})_{t\in\mathbb{R}_{+}}$-adapted
solution to \eqref{startrek}, by Theorem 7 on page 82 of \cite{krylov}.

\begin{assumption}\label{initial}
Let $\mathcal{G}_{\infty}$ be independent of $\mathcal{F}_{\infty}$.
Let $L_{0}=l(R,\mathbf{V}_{0},\mathbf{R}_{0})$ for some measurable
$l:[0,1]\times\mathcal{W}^{d\times d}\times\mathcal{W}^{d\times m}\to\mathbb{R}$ and $[0,1]$-uniformly distributed 
random variable $R$,
which is assumed to be independent of $\mathcal{G}_{\infty}\vee \mathcal{F}_{\infty}$.
\end{assumption}

\begin{remark} {\em An arbitrary joint law for $(L_{0},\mathbf{V}_{0},\mathbf{R}_{0})$
can be realized for suitable $l$, hence Assumption \ref{initial} is not restrictive at all.
For practical applications, actually, one may assume $L_{0}$ to be constant.}
\end{remark}

\begin{assumption}\label{thrice} The function $\zeta(x,v)$ is 
twice continuously differentiable in its first variable, $\partial_x\zeta,\partial_{xx}\zeta$ are bounded. 
Furthermore, there is $K>0$ such that $|\zeta(x,v_{1})-\zeta(x,v_{2})|\leq 
K(1+|v_{1}|+|v_{2}|)|v_{1}-v_{2}|$ for all $x\in\mathbb{R}^{d}$, $v_1,v_2\in\mathcal{V}$ (polynomial
Lipschitz condition in $v$). 
\end{assumption}

The following mean-reversion (or dissipativity) condition is 
rather standard, also in a non-Markovian context, see e.g.\ \cite{hairer-e}.

\begin{assumption}\label{dissipi} There exist $\alpha,\beta>0$, $\xi\geq 2$ such that
$$
\langle x,\zeta(x,v)\rangle\leq -\alpha |x|^{2}+\beta(1+|v|^{\xi}),
\ x\in\mathbb{R}^d,\ v\in\mathcal{V}.
$$
\end{assumption}

\begin{example}{\rm We briefly comment on the meaning of Assumptions \ref{thrice} and
\ref{dissipi} in a simple case with one asset ($d=1$) whose price satisfies 
$$
dS_{t} = \nu_{1}(S_{t})S_{t}\, dt+ V_{t}S_{t}\, d\overline{W}_{t}
$$
with some $S_{0}>0$, with a $(\mathbb{R}\setminus\{0\})\times (-1,1)$-valued
stationary process $(V_{t},\rho_{t})$ and 
Brownian motion $\overline{W}_{t}=\rho_t\, dB_t+\sqrt{1-\rho_t^2}\, dW_t$.
Let the function $\nu_{1}$ be such that $\bar{\nu}_{1}(x):=\nu_{1}(\exp(x))${}
is twice continuously differentiable with $\bar\nu_{1}'$,
$\bar\nu_{1}''$ bounded and satisfying $$
x\bar\nu_{1}(x)\leq -\bar{\alpha} |x|^{2}+\bar{\beta},\quad x\in\mathbb{R}
$$
with some $\bar{\alpha},\bar{\beta}>0$. 
Then $L_{t}:=\ln(S_{t})$ has dynamics
$$
dL_{t}=\left[\bar\nu_{1}(L_{t})+\frac{V_{t}^{2}}{2}\right]\, dt+ V_{t}\, d\overline{W}_{t}
$$
and $\zeta(x,v):=\bar\nu_{1}(x)+v^{2}/2$ satisfies Assumption \ref{dissipi} 
(with $\xi=2$ and with suitable $\alpha,\beta$).
Assumption \ref{thrice} also holds true. This example shows how the 
Lipschitz-continuity condition on $v$ naturally arises in Assumption \ref{thrice}. It also
shows that the most relevant case is where $\xi=2$.}
\end{example}

Finally, we stipulate moment conditions on the volatility process and on the initial condition.

\begin{assumption}\label{ve}
Let $E[|V_{0}|^{\max\{\xi,4\}}]<\infty$ holds for the $\xi$ of Assumption \ref{dissipi}.
\end{assumption}

\begin{assumption}\label{initi}
Let $E[|L_{0}|^2]<\infty$ hold.	
\end{assumption}

Our principal result is now presented.

\begin{theorem}\label{stability} Let Assumptions \ref{stationary}, \ref{initial}, 
\ref{thrice}, \ref{dissipi}, \ref{ve} and \ref{initi}
be in force. Then 
\begin{equation}\label{rarra}
\mathcal{L}(L_t,\mathbf{V}_{t},\mathbf{R}_{t})\to \mu_{\sharp},\ t\to\infty
\end{equation}
holds for some probability $\mu_{\sharp}$ on $\mathcal{B}(\mathbb{R}^{d}\times\mathcal{W}^{d\times d}\times
\mathcal{W}^{d\times m})$,
in the sense of weak convergence of probability measures. The probability
$\mu_{\sharp}$ does not depend on $L_0$ and it is invariant in the following sense:
if $\mathcal{L}(L_0,\mathbf{V}_{0},\mathbf{R}_{0})=\mu_{\sharp}$
then $\mathcal{L}(L_t,\mathbf{V}_{t},\mathbf{R}_{t})=\mu_{\sharp}$ for every $t>0$.
\end{theorem}

In the following theorem, instead of Assumption \ref{dissipi} one assumes
the weaker condition \eqref{madi} below. This comes at the price of strengthening Assumptions
\ref{ve} and \ref{initi} to \eqref{mocsing} below.

\begin{theorem}\label{stability1} Let Assumptions \ref{initial} and \ref{thrice} hold, let 
\begin{equation}\label{madi}
\langle x,\zeta(x,v)\rangle\leq -\alpha |x|^{1+\gamma}+\beta(1+|v|^{\xi}),
\ x\in\mathbb{R}^d,\ v\in\mathcal{V}
\end{equation}
hold for some $\alpha,\beta>0$, $\xi\geq 2$ and $0<\gamma<1$.
Let us assume \begin{equation}\label{mocsing}
E\left[\mathrm{e}^{\kappa_{0} |L_{0}|}\right]<\infty,\quad{}
E\left[\mathrm{e}^{\kappa_{0}  |V_{0}|^{\xi/\gamma}}\right]<\infty
\end{equation}
for some $\kappa_{0}>0$. Then the conclusions of Theorem \ref{stability} hold.
\end{theorem}

\section{Coupling constructions}\label{couplingsec}

Following the conventions
of measure theory, the total variation norm of a finite signed measure $\mu$ on
$\mathcal{B}(\mathcal{Z})$ is defined as
$$
||\mu||_{TV}:=\sup_{\phi\in\Phi_{1}}\left|\int_{\mathcal{Z}}\phi(z)\mu(dz)\right|,
$$ 
where $\Phi_{1}$ denotes the family of measurable functions $\phi:\mathcal{Z}\to [-1,1]$.
The underlying $\mathcal{Z}$ may vary but it will always be clear from the context.
Note that for $\mathcal{Z}$-valued random variables $Z_{1}$, $Z_{2}$
we always have
\begin{equation}\label{inez}
||\mathcal{L}(Z_{1})-\mathcal{L}(Z_{2})||_{TV}\leq 2P(Z_{1}\neq Z_{2}).
\end{equation}

\subsection{Markov chains}\label{marki}

First we will work in the setting of general state space discrete-time Markov chains. Our main
ideas will be explained in this simple context before turning to Markov chains in random
environments in the next subsection.

Proofs for the stochastic stability of Markov chains are usually based on two ingredients, see e.g.\ \cite{mt}. First, 
it is checked (using Lyapunov functions) that the chain returns often enough to a fixed set $C$.
Second, a minorization condition holds on $C$ for the transition kernel
so couplings occur whose probabilities can be estimated. Such $C$ are called ``small sets''.

When the state space is $\mathbb{R}^d$, it happens often 
that \emph{all} compact sets are small. This is the case for both discretized and
discretely sampled non-degenerate
diffusions. The coupling method of the present subsection exploits the 
latter property, formulated in more abstract terms. Otherwise we rely on standard 
``coupling from the past'' ideas, see e.g.\ \cite{propp-wilson,diaconis-freedman}.

Although Theorem \ref{maine} below seems to
be new, its statement contains little revelation. Its proof, on the contrary, presents original ideas
which will become fruitful in the more general setting
of the next subsection where existing results do not apply.
We will construct couplings on \emph{a sequence of} small sets
and then exploit (assuming a certain form of tightness) that
the chain stays in these sets with large enough probabilities. 
The crucial methodological contribution of this approach is that, instead of analysing 
return times to a set $C$ (which have a complicated dependence structure
due to the random environment), one can repeatedly use simple, one-step estimates. 

Another approach based on one-step estimates
was presented in \cite{hairer-mattingly}, using a contraction argument in a suitable
metric. When applying it in the presence of the random environment, however, the metric
to be used becomes dependent on that environment which sets limitations to the use 
of that method, see \cite{balazs}.

Let $\mathcal{X}$ be a Polish space. Let $Q(\cdot,\cdot)$
be a probabilistic kernel, i.e.\ $Q(\cdot,A)$ is measurable for each 
$A\in\mathcal{B}(\mathcal{X})$ and $Q(x,\cdot)$ is a probability law for each $x\in\mathcal{X}$. 
Let $X_t$, $t\in\mathbb{N}$ denote a Markov chain with transition kernel 
$Q$, started from some $X_0$.
We now define the set of initial laws starting from which the chain satisfies a tightness-like assumption.
We assume in the sequel that we are given a non-decreasing sequence of sets
$\mathcal{X}_n\in\mathcal{B}(\mathcal{X})$, $n\in\mathbb{N}$
with $\mathcal{X}_0\neq\emptyset$.

\begin{definition}
Let $\mathcal{P}_{b}$ denote the set of probabilities $\mu$ on $\mathcal{B}(\mathcal{X})$ such that
if $X_{0}$ has law $\mu$ then 
$$
\lim_{n\rightarrow\infty}\sup_{t\in\mathbb{N}}P(X_{t}\notin \mathcal{X}_{n}) = 0.
$$
\end{definition}

Notice that $\mathcal{P}_{b}$ might well be empty. We will write $X_{0}\in\mathcal{P}_{b}$ when 
we indeed mean $\mathcal{L}(X_{0})\in\mathcal{P}_{b}$.
We stipulate next that minorization conditions should hold on \emph{each} of the sets $\mathcal{X}_{n}$.

\begin{assumption}\label{minor}
There exists a sequence $\alpha_{n}\in (0,1]$, $n\in\mathbb{N}$ and
a sequence of probability measures $\nu_{n}$, $n\in\mathbb{N}$ such that
\begin{equation}\label{www}
Q(x,A)\geq \alpha_{n}\nu_n(A),\ A\in\mathcal{B}(\mathcal{X}),\ x\in\mathcal{X}_n,\ n\in\mathbb{N}.
\end{equation}
\end{assumption}

We recall a representation result for kernels satisfying the minorization condition \eqref{www}, in terms of random
mappings that are constant on the respective $\mathcal{X}_{n}$ with probability at least $\alpha_{n}$.

\begin{lemma}\label{indi} Let Assumption \ref{minor} be in force. Let $\mathbf{U}$ be
a uniform random variable on $[0,1]$. For each 
$n\in\mathbb{N}$, there exists a mapping $T^{n}(\cdot,\cdot):[0,1]\times\mathcal{X}\to\mathcal{X}$
satisfying 
\begin{equation*}
Q(x,A)=P(T(\mathbf{U},x)\in A),\ x\in\mathcal{X},\ A\in\mathcal{B}(\mathcal{X}),
\end{equation*}
such that for all $u\in [0,\alpha_{n}]$,
\begin{equation}\label{ultramarin}
T^{n}(u,x_1)=T^{n}(u,x_2)\quad\mbox{for all}\quad x_1,x_2\in\mathcal{X}_n.
\end{equation}
\end{lemma}
\begin{proof} Such a representation is well-known, see page 228 in \cite{bwbook}. 
\end{proof}

\begin{theorem}\label{maine} Let Assumption \ref{minor} hold. 
Then there exists a probability $\mu_{*}$ on $\mathcal{B}(\mathcal{X})$ such that
\begin{equation}
||\mathcal{L}(X_t)- \mu_{*}||_{TV}\to 0,\ t\to\infty
\end{equation}
holds for every $X_{0}\in\mathcal{P}_{b}$.
\end{theorem}
\begin{proof} Theorem \ref{maine} follows from Theorem \ref{maine1} below (choosing $\mathcal{Y}$ 
a singleton). Nonetheless we provide a proof in the present, simple
setting, in order to elucidate the main ideas.

Fix $\varepsilon>0$ and choose $n=n(\varepsilon)$ so large that
\begin{equation}\label{malriv0}
\sup_{t\in\mathbb{N}}P(X_{t}\notin\mathcal{X}_{n})
\leq \varepsilon.
\end{equation}
We will estimate coupling probabilities on $\mathcal{X}_{n}$, using independent copies of the random
mappings constructed in Lemma \ref{indi} above.

Let $\mathbf{U}_{k}$, $k\in -\mathbb{N}$ be an 
independent sequence of uniform random variables on $[0,1]$,
independent of $X_{0}$. Let $T^{n}(\cdot,\cdot)$ be the mapping constructed in Lemma \ref{indi}.
Define the process
$$
\tilde{X}_{t}:=[T^{n}(\mathbf{U}_{0},\cdot)\circ\cdots\circ T^{n}(\mathbf{U}_{-t+1},\cdot)](X_{0}),
\ t\in\mathbb{N}
$$
where we mean $\tilde{X}_{0}=X_{0}$.
Notice that $\mathcal{L}(\tilde{X}_{t})=\mathcal{L}(X_{t})$ for each $t\in\mathbb{N}$.

Fix integers $1\leq s<t$. For each $j=0,\ldots,s$, define the following 
disjoint events:
\begin{eqnarray*}
A^{s,t}_{j} &:=& \left\{
[T^{n}(\mathbf{U}_{-j},\cdot)\circ\cdots\circ 
T^{n}(\mathbf{U}_{-t+1},\cdot)](X_{0}) =  [T^{n}(\mathbf{U}_{-j})\circ\ldots\circ T^{n}(\mathbf{U}_{-s+1},\cdot)]
(X_{0})\right\},\\
B^{s,t}_{j} &:=& \left\{
[T^{n}(\mathbf{U}_{-j},\cdot)\circ\cdots\circ 
T^{n}(\mathbf{U}_{-t+1},\cdot)](X_{0})
\neq [T^{n}(\mathbf{U}_{-j},\cdot)\circ\ldots\circ T^{n}(\mathbf{U}_{-s+1},\cdot)]
(X_{0}),\right.\\
& & \left. [T^{n}(\mathbf{U}_{-j},\cdot)\circ\cdots\circ 
T^{n}(\mathbf{U}_{-t+1},\cdot)](X_{0})\in\mathcal{X}_{n},
\ [T^{n}(\mathbf{U}_{-j},\cdot)\circ\ldots\circ T^{n}(\mathbf{U}_{-s+1},\cdot)]
(X_{0})\in\mathcal{X}_{n}\right\},\\
C^{s,t}_{j} &:=& \Omega\setminus (A^{s,t}_{j}\cup B^{s,t}_{j}),
\end{eqnarray*}
where we mean 
\begin{eqnarray*}
A^{s,t}_{s} &:=& \left\{
[T^{n}(\mathbf{U}_{-s},\cdot)\circ\cdots\circ 
T^{n}(\mathbf{U}_{-t+1},\cdot)](X_{0}) = X_{0}\right\},\\
B^{s,t}_{s} &:=& \left\{
[T^{n}(\mathbf{U}_{-s},\cdot)\circ\cdots\circ 
T^{n}(\mathbf{U}_{-t+1},\cdot)](X_{0}) \neq X_{0},\ [T^{n}(\mathbf{U}_{-s},\cdot)\circ\cdots\circ 
T^{n}(\mathbf{U}_{-t+1},\cdot)](X_{0})\in\mathcal{X}_{n},\ X_{0}\in\mathcal{X}_{n}\right\}.
\end{eqnarray*}

Define also $p_{j}^{s,t}:=P(A^{s,t}_{j})$.
We aim to show that, for $s$ large, $p^{s,t}_{0}$ is close to $1$ for each $t>s$, which
means that $\tilde{X}_{t}$ very likely equals $\tilde{X}_{s}$.
We will estimate $p_{j}^{s,t}$ by backward recursion.
Notice that 
\begin{eqnarray}
P(C^{s,t}_{j})
\nonumber &\leq& P([T^{n}(\mathbf{U}_{-j},\cdot)\circ\cdots\circ 
T^{n}(\mathbf{U}_{-t+1},\cdot)](X_{0})\notin\mathcal{X}_{n})
+ P([T^{n}(\mathbf{U}_{-j},\cdot)\circ\ldots\circ T^{n}(\mathbf{U}_{-s+1},\cdot)]
(X_{0})\notin\mathcal{X}_{n})\\
&=& P(X_{t-j}\notin\mathcal{X}_{n})+P(X_{s-j}\notin\mathcal{X}_{n})\leq  2\varepsilon,\label{montreux0}	
\end{eqnarray}
by \eqref{malriv0}.
Define $\mathcal{H}_{j,t}:=\sigma(X_{0},\mathbf{U}_{-j},\ldots,\mathbf{U}_{-t+1})$. On  
the event $B_{j}^{s,t}\in\mathcal{H}_{j,t}$ we have
\begin{eqnarray*}
P\left(A_{j-1}^{s,t}\mid\mathcal{H}_{j,t}\right)
\geq P\left(\mathbf{U}_{-j+1}\in [0,\alpha_{n}]\mid\mathcal{H}_{j,t}\right)={}
P(\mathbf{U}_{-j+1}\in [0,\alpha_{n}])=\alpha_{n}\mbox{ a.s.}
\end{eqnarray*}
since $T^{n}(\mathbf{U}_{-j+1},\cdot)$ is a constant mapping on $\mathcal{X}_{n}$ 
when $\mathbf{U}_{-j+1}^{1}\in [0,\alpha_{n}]$, and $\mathbf{U}_{-j+1}$ is independent of 
$\mathcal{H}_{j,t}$. On the other hand, on the event $A_{j}^{s,t}\in \mathcal{H}_{j,t}$ we have
$P\left(A_{j-1}^{s,t}\vert\mathcal{H}_{j,t}\right)=1$ a.s.\ for trivial reasons.
Hence
\begin{equation}\label{mo0}
p_{j-1}^{s,t}\geq p_{j}^{s,t}+\alpha_{n}P(B_{j}^{s,t})\geq 
p_{j}^{s,t}+\alpha_{n}(1-p_{j}^{s,t}-2\varepsilon), 
\end{equation}
using \eqref{montreux0}. We get by backward recursion using \eqref{mo0}, starting from the trivial $p_{s}^{s,t}\geq 0$, 
that
$$
p_{0}^{s,t}\geq (1-2\varepsilon)\alpha_{n}\frac{1-(1-\alpha_{n})^{s}}{1-(1-\alpha_{n})}=(1-2\varepsilon){}
[1-(1-\alpha_{n})^{s}],
$$
remembering also the formula for the sum of a geometric series.
It follows from \eqref{inez} that for all integers $1\leq s <t$,
\begin{equation}\label{walmart}
||\mathcal{L}(X_{t})-\mathcal{L}(X_{s})||_{TV}\leq 2P(\tilde{X}_{t}\neq\tilde{X}_{s})
=2(1-p_{0}^{s,t})\leq 4\varepsilon +2(1-\alpha_{n})^{s},
\end{equation}
which is smaller than $5\varepsilon$ for $s$ large enough.
As $\varepsilon$ was arbitrary, the sequence $\mathcal{L}(X_{t})$, $t\in\mathbb{N}$ is shown
to be Cauchy for the total variation distance hence it converges to some probability $\mu_{*}$.

Let $X_t$, $X_t'$, $t\in\mathbb{N}$ denote Markov chains with transition kernel 
$Q$, started from $X_0, X_0'\in\mathcal{P}_{b}$, respectively.
Then, using $\mathbf{U}_{k}$, $k\in -\mathbb{N}$ independent of $\sigma(X_{0},X_{0}')$,
we get $||\mathcal{L}(X_t)- \mathcal{L}(X_t')||_{TV}\to 0$ as $t\to\infty$ analogously
to the argument above. This shows that $\mu_{*}$ is independent of the choice of $X_{0}\in\mathcal{P}_{b}$.
\end{proof}

\begin{remark}\label{rd} {\rm Assume $\mathcal{X}:=\mathbb{R}^{d}$ and $\mathcal{X}_{n}:=\{x\in\mathcal{X}:
|x|\leq n\}$, $n\in\mathbb{N}$. Let $V(x):=g(|x|)$ for some non-decreasing $g:\mathbb{R}_{+}\to\mathbb{R}_{+}$
with $g(\infty)=\infty$.
If the initial state $X_{0}$ is such that
\begin{equation}\label{vveges}
\sup_{k\in\mathbb{N}}E[V(X_{k})]<\infty,
\end{equation}
then $X_{0}\in\mathcal{P}_{b}$, as seen from Markov's inequality.}
\end{remark}

The argument for proving Theorem \ref{maine} above, in fact, provides us with a convergence rate estimate, too.
For each $t$, \eqref{roro} below allows to optimize over $n$ and to choose $n=n(t)$ that gives the best
estimate.
   
\begin{corollary}\label{geany} Under Assumption \ref{minor}, in the setting of Remark \ref{rd},
for each $n\in\mathbb{N}$ and $t\in\mathbb{N}$,	
\begin{equation}\label{roro}
||\mathcal{L}(X_{t})-\mu_{*}||_{TV}\leq 
4\frac{\sup_{k\in\mathbb{N}}E[V(X_{k})]}{g(n)}+2(1-\alpha_{n})^{t}.
\end{equation}
\end{corollary}
\begin{proof}
This follows from \eqref{malriv0}, \eqref{walmart} and from Markov's inequality.
\end{proof}

We demonstrate the application of Corollary \ref{geany} and the resulting rate through a simple example.
  
\begin{example}\label{intoto}
{\rm Consider a stable scalar AR(1) process, where $\cX=\mR$ and the dynamics is
    \begin{equation}
      \label{eq:ardef}
      X_{t+1} = \gamma X_t + \eps_{t+1},
    \end{equation}
    where $0<\gamma<1$, $\eps_t$ is an independent series of standard
    Gaussian variables, and $X_0$ is a constant initialization.

  In order to apply Corollary \ref{geany}, we choose
  $V(x)=g(|x|)=e^{\beta x^2}$ with $\beta<\frac{1-\gamma^2}{2}$.
  To confirm \eqref{vveges}, expanding the dynamics equation \eqref{eq:ardef} we see
  \begin{equation*}
    X_t = \gamma^t X_0 + \sum_{s=1}^t \gamma^{t-s}\eps_s \quad \sim
    \quad \mathcal{N}\left(\gamma^t X_0, \frac{1-\gamma^{2t}}{1-\gamma^2}\right).
  \end{equation*}
  Consequently,
  \begin{align*}
    EV(X_t) &=
              \frac{1}{\sqrt{2\pi\frac{1-\gamma^{2t}}{1-\gamma^2}}}\int_{-\infty}^\infty
              e^{-\frac{1-\gamma^{2}}{2(1-\gamma^{2t})}(z-\gamma^t
              X_0)^2}e^{\beta z^2} dz\\
            &\le \frac{1}{\sqrt{2\pi}} \int_{-\infty}^\infty
              e^{-\frac{1-\gamma^{2}}{2}(z-\gamma^t
              X_0)^2}e^{\beta z^2} dz < \infty,
  \end{align*}
  and this quantity is also bounded above uniformly in $t$ by some 
  $c(\gamma,\beta,X_0)$ since $|\gamma^t X_0|$ decreases as
  $t\to\infty$.

  We also need Assumption \ref{minor}, the minorization condition for a sequence of small sets. Let
  \begin{equation*}
    \cX_n=[-n,n], \qquad \nu=\frac{1}{2}Leb\vert_{[-1,1]},
  \end{equation*}
  for all $n$. In order to acquire
  $\alpha_n$, we need to find the infimum of $\frac{dQ(x,\cdot)}{d\nu(\cdot)}$
  on the appropriate sets, and now that they are both absolutely
  continuous distributions, this boils down to comparing the densities, therefore
  \begin{equation}
    \label{eq:aralpha}
    \alpha_n=\inf_{x\in [-n,n],z\in [-1,1]}
    \frac{Q(x,dz)}{ \frac{1}{2}dz } =
    \sqrt{\frac{2}{\pi}}e^{-\frac{(\gamma n + 1)^2}{2}}.
  \end{equation}

  Substituting the computed expressions Corollary \ref{geany} provides
  \begin{equation}
    \label{eq:armainbound}
    ||\mathcal{L}(X_{t})-\mu_{*}||_{TV}\leq
    \frac{4c(\gamma,\beta,X_0)}{\exp(\beta
      n^2)}
    + 2\left(1-\sqrt{\frac{2}{\pi}}e^{-\frac{(\gamma n + 1)^2}{2}}\right)^t.  
  \end{equation}
  It remains to choose $n$ depending on $t$ to get the best bound
  possible. Clearly there is a tradeoff: for small values of $n$, the
  first term is weak while for large values of $n$ the
  second term increases and can remain bounded away from $0$.

  Let us present the heuristics to find a near-optimal $n$.
  The second term in \eqref{eq:armainbound} is approximately
  \begin{equation*}
    2\exp\left(-t \sqrt\frac 2 \pi \exp\left(-\frac{(\gamma n + 1)^2}{2}\right)\right).
  \end{equation*}
  We get the optimal bounds if the two terms agree
  (ignoring constants):
  \begin{align*}
    \exp(-\beta n^2) &= \exp\left(-t \sqrt\frac 2 \pi \exp\left(-\frac{(\gamma n + 1)^2}{2}\right)\right),\\
%    -\beta n^2 &= -t \sqrt\frac 2 \pi \exp\left(-\frac{(\gamma n + 1)^2}{2}\right),\\
    \log \beta + 2\log n &= \log t + \frac 1 2 \log \frac 2 \pi - \frac{(\gamma n + 1)^2}{2}.
  \end{align*}
  It is easy to see that the value of $\frac{\sqrt{2\log t}}{\gamma}$
  is slightly too high for $n$. Still, inspired by this option we
  choose
  \begin{equation*}
    n=\left\lceil\left(\frac{\sqrt{2}}{\gamma}-\eta\right)\sqrt{\log t}\right\rceil
  \end{equation*}
  with some small $\eta>0$.
  Using this choice in our bound \eqref{eq:armainbound} and noting
  $$
  (\gamma n +1)^{2}\leq \left(\gamma\left(\frac{\sqrt{2}}{\gamma}-\eta\right)\sqrt{\log t}+2\right)^{2}
  $$
  we get
  \begin{equation*}
    ||\mathcal{L}(X_{t})-\mu_{*}||_{TV}\leq
    \frac{4c(\gamma,\beta,X_0)}{\exp\left(\beta\left(\frac{\sqrt{2}}{\gamma}-\eta\right)^2\log t\right)}
    + 2\left(1-\sqrt{\frac 2
        \pi}\exp\left[-\left(\left(1-\frac{\gamma\eta}{\sqrt{2}}\right)\sqrt{\log t} + \sqrt{2}\right)^2\right]\right)^t.
  \end{equation*}
  In the exponent of the first term we could choose the coefficient of the logarithm
  arbitrarily close to $\frac{1-\gamma^2}{2}
  \left(\frac{\sqrt{2}}{\gamma}\right)^2=\frac{1}{\gamma^2}-1$. Although
  the second term looks daunting, observe that it has the order of
  $(1-t^{-1+\eta'})^t$ with some $\eta'>0$ therefore it has
  subpolynomial decay and is negligible compared to the first term.

  Summing up, for a rate estimate we get that for any $h>0$ there is some
  constant $C_h>0$ such that
  \begin{equation}
    \label{eq:arfinal}
    ||\mathcal{L}(X_{t})-\mu_{*}||_{TV}\leq \frac{C_h}{t^{\frac{1}{\gamma^2}-1-h}}.
  \end{equation}
In the model \eqref{eq:ardef}, $||\mathcal{L}(X_{t})-\mu_{*}||_{TV}$
decreases geometrically in $t$ so only a suboptimal rate can be achieved by our method. 
Nevertheless, the estimates leading to \eqref{eq:arfinal} are of great interest 
since they can serve as a basis for similar results for certain non-Markovian models,
where power convergence rates are common, see e.g.\ \cite{hairer-e}.
One can thus treat models like 
\eqref{eq:disc-logvol} below (which are not covered by current literature).
Then, using technology from \cite{balazs,attila},
various mixing properties and laws of large numbers (with rate estimates) 
can be established for functionals of the process $X_{t}$, $t\in\mathbb{N}$. 
Central limit theorems can also be derived from
mixing conditions, see \cite{alfa}. These developments, however, are out of the
scope of the present article.}
%Working out these ideas in detail, however, requires substantial further effort.}
\end{example}

\subsection{Markov chains in random environments}\label{nonmarki}

We now extend Theorem \ref{maine} to Markov chains in random environments.
These processes will still evolve in $\mathcal{X}=\cup_{n\in\mathbb{N}}\mathcal{X}_{n}$ but
their dynamics will be influenced by another random process we are just about to introduce.
Let $\mathcal{Y}$ be another Polish space and let $Y_t$, $t\in\mathbb{Z}$ be a (strict sense)
stationary process in $\mathcal{Y}$. We assume that a non-decreasing sequence 
$\mathcal{Y}_{n}\in\mathcal{B}(\mathcal{Y})$, $n\in\mathbb{N}$ is given with $\mathcal{Y}_{0}\neq\emptyset$.{}
% such that 
%$\cup_{n\in\mathbb{N}}\mathcal{Y}_{n}=\mathcal{Y}$.
Let $Q:\mathcal{X}\times\mathcal{Y}\times\mathcal{B}(\mathcal{X})\to [0,1]$ 
be a parametrized family of transition kernels, i.e.\ $Q(\cdot,\cdot,A)$ is measurable for all $A\in\mathcal{B}(\mathcal{X})$
and $Q(x,y,\cdot)$ is a probability for all $(x,y)\in\mathcal{X}\times\mathcal{Y}$.
We say that the process ${X}_t$, $t\in\mathbb{N}$ is a Markov chain in a random environment
with transition kernel $Q$ if it is an $\mathcal{X}$-valued stochastic
process such that 
\begin{equation}\label{recu}
P({X}_{t+1}\in A\mid \sigma(Y_j,\ j\in\mathbb{Z};\ X_j,\ 0\leq j\leq t))=Q(X_t,Y_{t},A),\ t\in\mathbb{N}. 
\end{equation}

Denote by $\mathcal{M}_{0}$ the set of probability laws on $\mathcal{X}\times\mathcal{Y}^{\mathbb{Z}}$
such that their second marginal equals the law of $(Y_{k})_{k\in\mathbb{Z}}$. Let $\mathcal{M}_{b}$
denote the set of those $\mu\in\mathcal{M}_{0}$ for which the process $X_{t}$, $t\in\mathbb{N}$ started from $X_{0}$
with $\mathcal{L}(X_{0},(Y_{k})_{k\in\mathbb{Z}})=\mu$ satisfies
\begin{equation}\label{tighti}
\sup_{t\in\mathbb{N}}P(X_{t}\notin\mathcal{X}_{n})\to 0,\ n\to\infty.
\end{equation}
We will write $X_{0}\in\mathcal{M}_{b}$ in the sequel when we really mean
$\mathcal{L}(X_{0},(Y_{k})_{k\in\mathbb{Z}})\in\mathcal{M}_{b}$.

\begin{assumption}\label{minor1}
Let $P(Y_{0}\notin\mathcal{Y}_{n})\to 0$ hold as $n\to\infty$. There exists a sequence $\alpha_{n}\in (0,1]$, 
$n\in\mathbb{N}$ and
a sequence of probability measures $\nu_{n}$, $n\in\mathbb{N}$ such that for all $n\in\mathbb{N}$, 
$$
Q(x,y,A)\geq \alpha_{n}\nu_n(A),\ A\in\mathcal{B}(\mathcal{X}),\ y\in\mathcal{Y}_{n},\ x\in\mathcal{X}_n. 
$$
\end{assumption}

A parametric version of Lemma \ref{indi} comes next.

\begin{lemma}\label{indi1} Let Assumption \ref{minor1} be in force. Let $\mathbf{U}$ be
a uniform random variable on $[0,1]$. For each 
$n\in\mathbb{N}$, there exists a measurable mapping $T^{n}(\cdot,\cdot,\cdot):[0,1]\times\mathcal{X}\times\mathcal{Y}\to\mathcal{X}$
satisfying 
$Q(x,y,A)=P(T^{n}(\mathbf{U},x,y)\in A)$, $x\in\mathcal{X}$, $y\in\mathcal{Y}$, $A\in\mathcal{B}(\mathcal{X})$
such that for all $u\in [0,\alpha_{n}]$,
$$
T^{n}(u,x_1,y)=T^{n}(u,x_2,y)\quad\mbox{for all}\quad x_1,x_2\in\mathcal{X}_n,\ y\in\mathcal{Y}_{n}.
$$
\end{lemma}
\begin{proof}
This is a straightforward extension of the case with $\mathcal{Y}$ a singleton, that is, of Lemma
\ref{indi} above. See Lemma 7.1 of \cite{attila}.
\end{proof}

The following abstract result serves as the basis of Section \ref{stabi} below. We do not know of
any similar results in the literature. Existing papers have fairly restrictive assumptions:
either Doeblin-like conditions (as in \cite{kifer1,kifer2,sep}) or strong contractivity hypotheses (as in \cite{stenflo}).

\begin{theorem}\label{maine1} 
Let Assumption \ref{minor1} hold and let $\mathcal{M}_{b}\neq\emptyset$. Let 
$X_t$, $t\in\mathbb{N}$ denote a Markov chain in a random environment
with transition kernel $Q$, started from some $X_0\in\mathcal{M}_{b}$.
Then
there exists a probability $\mu_{\sharp}$ on $\mathcal{B}(\mathcal{X}\times\mathcal{Y}^{\mathbb{N}})$ such that
\begin{equation}\label{holes}
||\mathcal{L}(X_{t},(Y_{t+k})_{k\in\mathbb{Z}})-\mu_{\sharp}||_{TV}\to 0,\ t\to\infty.{}
\end{equation}
If $X_{t}'$, $t\in\mathbb{N}$ is another such Markov chain in random environment started from 
$X_{0}'\in\mathcal{M}_{b}$ then
\begin{equation}\label{trolls}
||\mathcal{L}(X_{t},(Y_{t+k})_{k\in\mathbb{Z}})- \mathcal{L}(X'_{t},(Y_{t+k})_{k\in\mathbb{Z}})||_{TV}\to 0,\ t\to\infty.
\end{equation}
In particular, $\mu_{\sharp}$ does not depend on the choice of $X_{0}\in\mathcal{M}_{b}$. 
The probability $\mu_{\sharp}$ is invariant in the following sense:
if $X_{0}$ is such that $\mathcal{L}(X_{0},(Y_{k})_{k\in\mathbb{Z}})=\mu_{\sharp}$ then 
$\mathcal{L}(X_{t},(Y_{t+k})_{k\in\mathbb{Z}})=\mu_{\sharp}$ for each $t\in\mathbb{N}$.
\end{theorem}

\begin{proof} The core idea of the proof is identical to that of Theorem \ref{maine},
with the extra task of checking whether the process $Y$ stays in $\mathcal{Y}_{n}$ for some
suitable $n$.
In order to prove invariance, however, here we need to construct 
$\tilde{X}_{\infty}$ such that $\tilde{X}_{t}$ (to be defined soon) converges to $\tilde{X}_{\infty}$ a.s.\ in a
stationary way (along a suitable subsequence). This requires a more complicated setup. 

There exists a measurable function $g:\mathcal{Y}^{\mathbb{Z}}\times [0,1]\to \mathcal{X}$
and a uniform $[0,1]$-valued random variable $R$, independent of 
$\sigma(Y_{k},{k\in\mathbb{Z}})$, such that $\mathcal{L}(X_{0},(Y_{k})_{k\in\mathbb{Z}})=
\mathcal{L}(g((Y_{k})_{k\in\mathbb{Z}},R),(Y_{k})_{k\in\mathbb{Z}})$.
Let $\mathbf{U}_{k}$, $k\in -\mathbb{N}$ be an 
independent family of uniform random variables on $[0,1]$,
independent of $\sigma(R,(Y_{k})_{k\in\mathbb{Z}})$.
Let $T^{n}(\cdot,\cdot,\cdot)$, $n\in\mathbb{N}$ be the mappings constructed in Lemma \ref{indi1}.

For each integer $m\geq 1$ choose $n(m)\in\mathbb{N}$ so large that
\begin{equation}\label{malriv}
P(Y_{0}\notin\mathcal{Y}_{n(m)})+\sup_{k\in\mathbb{N}}P(X_{k}\notin\mathcal{X}_{n(m)})
\leq 1/2^{m}.
\end{equation}

Let $N(m)\geq 1$ be so large that $(1-\alpha_{n(m)})^{N(m)}\leq 1/2^{m}$.
Define $M_{0}:=0$, 
$M_{m}:=\sum_{j=1}^{m}N(j)$. Define the following random mappings from $\mathcal{X}\to\mathcal{X}$, for each $m\geq 1$:
$$
\tilde{T}_{m}(\cdot):=T^{n(m)}(\mathbf{U}_{-M_{m-1}},\cdot,Y_{-M_{m-1}-1})\circ \ldots \circ T^{n(m)}
(\mathbf{U}_{-M_{m}+1},\cdot,Y_{-M_{m}})
$$
and
$$
\mathbf{T}_{m}(\cdot):=\tilde{T}_{1}(\cdot)\circ\ldots\circ \tilde{T}_{m}(\cdot).
$$
Let $\mathbf{T}_{0}$ be the identity mapping of $\mathcal{X}$.

Let $\tilde{X}_{0}:=g((Y_{k})_{k\in\mathbb{Z}},R)$ and
for each $m\in\mathbb{N}$ and each $M_{m}+1\leq t\leq M_{m+1}$, define the process
$$
\tilde{X}_{t}:=\mathbf{T}_{m}(\cdot)\circ T^{n(m+1)}(\mathbf{U}_{-M_{m}},\cdot,Y_{-M_{m}-1})\circ\ldots\circ 
T^{n(m+1)}(\mathbf{U}_{-t+1},\cdot,{}
Y_{-t})(g((Y_{-t+k})_{k\in\mathbb{Z}},R)).
$$
Notice that $\mathcal{L}(\tilde{X}_{t},(Y_{k})_{k\in\mathbb{Z}})=\mathcal{L}(X_{t},(Y_{t+k})_{k\in\mathbb{Z}})$
by construction, for each $t\in\mathbb{N}$.

Fix $m\geq 2$ and let $M_{m}+1\leq t\leq M_{m+1}$ be arbitrary. 
For each $j=M_{m-1},\ldots,M_{m}$ we will define the following 
random variables:
\begin{eqnarray*}
V_{j,t} &:=& [T^{n(m)}(\mathbf{U}_{-j},\cdot,Y_{-j-1})\circ\cdots\circ 
T^{n(m)}(\mathbf{U}_{-M_{m}+1},\cdot,Y_{-M_{m}})\circ
T^{n(m+1)}(\mathbf{U}_{-M_{m}},\cdot,Y_{-M_{m}-1})\circ \\ 
&\cdots& \circ 
T^{n(m+1)}(\mathbf{U}_{-t+1},\cdot,Y_{-t})](g((Y_{-t+k})_{k\in\mathbb{Z}},R)),\\
W_{j,t} &:=&
[T^{n(m)}(\mathbf{U}_{-j},\cdot,Y_{-j-1})\circ\cdots\circ T^{n(m)}(\mathbf{U}_{-M_{m}+1},\cdot,Y_{-M_{m}})]
(g((Y_{-M_{m}+k})_{k\in\mathbb{Z}},R)),
\end{eqnarray*}
with the understanding that
$$
W_{M_{m},t}=g((Y_{-M_{m}+k})_{k\in\mathbb{Z}},R)
$$
and
$$
V_{M_{m},t}:= T^{n(m+1)}(\mathbf{U}_{-M_{m}},\cdot,Y_{-M_{m}-1})\circ\cdots\circ 
T^{n(m+1)}(\mathbf{U}_{-t+1},\cdot,Y_{-t})(g((Y_{-t+k})_{k\in\mathbb{Z}},R)).
$$
Consider the corresponding 
disjoint events
\begin{eqnarray*}
A_{j,t} &:=& \left\{ V_{j,t}=W_{j,t}\right\},\\
B_{j,t} &:=& \left\{ V_{j,t}\neq W_{j,t},
V_{j,t}\in\mathcal{X}_{n(m)},\ W_{j,t}\in\mathcal{X}_{n(m)},\ Y_{-j}\in\mathcal{Y}_{n(m)} \right\},\\
C_{j,t} &:=& \Omega\setminus (A_{j,t}\cup B_{j,t}).
\end{eqnarray*}
Define also $p_{j,t}:=P(A_{j,t})$, $j=M_{m-1},\ldots,M_{m}$.
Notice that 
\begin{eqnarray}
P(C_{j,t})
\nonumber &\leq& P(V_{j,t}\notin\mathcal{X}_{n(m)})+
P(W_{j,t}\notin\mathcal{X}_{n(m)}) +P(Y_{-j}\notin\mathcal{Y}_{n(m)})\\
\nonumber &=& P(X_{t-j}\notin\mathcal{X}_{n(m)})+P(X_{M_{m}-j}\notin\mathcal{X}_{n(m)})
+P(Y_{-j}\notin\mathcal{Y}_{n(m)})\\
&\leq & 1/2^{m-1},\label{montreux}	
\end{eqnarray}
by the stationarity of the process $Y$ and by \eqref{malriv}.

Define $\mathcal{H}_{j,t}:=\sigma((Y_{k})_{k\in\mathbb{Z}},R,\mathbf{U}_{-j},\ldots,\mathbf{U}_{-t+1})$. On  
$B_{j,t}\in\mathcal{H}_{j,t}$ we have
\begin{eqnarray*}
P\left(A_{j-1,t}\mid\mathcal{H}_{j,t}\right)
\geq P\left(\mathbf{U}_{-j+1}\in [0,\alpha_{n(m)}]\mid\mathcal{H}_{j,t}\right)={}
P(\mathbf{U}_{-j+1}\in [0,\alpha_{n(m)}])=\alpha_{n(m)}\mbox{ a.s.}
\end{eqnarray*}
since $T^{n(m)}(\mathbf{U}_{-j+1},\cdot,y)$ is a constant mapping on $\mathcal{X}_{n(m)}$, for each 
$y\in\mathcal{Y}_{n(m)}$
when $\mathbf{U}_{-j+1}\in [0,\alpha_{n(m)}]$, and $\mathbf{U}_{-j+1}$ is independent of 
$\mathcal{H}_{j,t}$. On the other hand, on $A_{j,t}\in \mathcal{H}_{j,t}$ we have
$P\left(A_{j-1,t}\vert\mathcal{H}_{j,t}\right)=1$ a.s., trivially.
Hence
\begin{equation}\label{mo}
p_{j-1,t}\geq p_{j,t}+\alpha_{n(m)}P(B_{j,t})\geq 
p_{j,t}+\alpha_{n(m)}(1-p_{j,t}-1/2^{m-1}), 
\end{equation}
using \eqref{montreux}, which leads (by backward induction starting from $p_{M_{m},t}\geq 0$) to
$$
p_{M_{m-1},t}\geq (1-1/2^{m-1}){}
[1-(1-\alpha_{n(m)})^{N(m)}],
$$
and eventually to
\begin{eqnarray}\label{molise1}
P(\tilde{X}_{t}\neq\tilde{X}_{M_{m}})\leq P(V_{M_{m-1},t}\neq W_{M_{m-1},t})=
1-p_{M_{m-1},t}\leq 1/2^{m-1}+ (1-\alpha_{n(m)})^{N(m)}\leq 1/2^{m-2},
\end{eqnarray}
remembering the choice of $N(m)$. 
These relations establish, in particular, that for the event 
$$
A_{m}:=\left\{\tilde{X}_{M_{j}}=\tilde{X}_{M_{m}}\mbox{ for all }j\geq m\right\},
$$
we have
\begin{equation}\label{molise2}
P(\Omega\setminus A_{m})\leq \sum_{j=m} \frac{1}{2^{j-2}}\leq 1/2^{m-3}.
\end{equation}
We can thus define unambiguously $\tilde{X}_{\infty}:=\tilde{X}_{M_{m}}$ on $A_{m}$ 
and, doing this for all $m\geq 2$, a random variable $\tilde{X}_{\infty}$ gets almost surely defined. 
Clearly, for all $M_{m}+1\leq t\leq M_{m+1}$,
$$
P(\tilde{X}_{t}\neq\tilde{X}_{\infty})\leq 
P(\tilde{X}_{t}\neq\tilde{X}_{M_{m}})+P(\tilde{X}_{M_{m}}\neq\tilde{X}_{\infty})
\leq 1/2^{m-4},
$$
by \eqref{molise1} and \eqref{molise2}.
Denoting by $\mu_{\sharp}$ the law of $(\tilde{X}_{\infty},(Y_{k})_{k\in\mathbb{Z}})$, 
$$
||\mathcal{L}(X_{t},(Y_{t+k})_{k\in\mathbb{Z}})-\mu_{\sharp}||_{TV}\leq 2P(\tilde{X}_{t}\neq\tilde{X}_{\infty})\to 0,
\ t\to\infty.{}
$$

Now we turn to proving \eqref{trolls}. In addition to $\tilde{X}_{t}$, let us also define $\tilde{X}_{t}'$
in the same manner with $g$ replaced by $g':\mathcal{Y}^{\mathbb{Z}}\times [0,1]\to \mathcal{X}$ such that
$\mathcal{L}(X_{0}',(Y_{k})_{k\in\mathbb{Z}})=
\mathcal{L}(g'((Y_{k})_{k\in\mathbb{Z}},R),(Y_{k})_{k\in\mathbb{Z}})$.
We get by analogous arguments that
$$
||\mathcal{L}({X}_{t},(Y_{t+k})_{k\in\mathbb{Z}})-\mathcal{L}({X}_{t}',(Y_{t+k})_{k\in\mathbb{Z}})||_{TV}=
||\mathcal{L}(\tilde{X}_{t},(Y_{k})_{k\in\mathbb{Z}})-\mathcal{L}(\tilde{X}_{t}',(Y_{k})_{k\in\mathbb{Z}})||_{TV}
\to 0,\ t\to\infty.
$$

{}
To see invariance, fix $\varepsilon>0$ and notice that for $m=m(\varepsilon)$ large enough,
\begin{equation}\label{fors}
P(\tilde{X}_{M_{m}}\neq \tilde{X}_{\infty})+P(\tilde{X}_{M_{m}+1}\neq \tilde{X}_{\infty})\leq \varepsilon.
\end{equation}
Let us take $\mathbf{U}^{*}$ uniform on $[0,1]$, independent of all the random objects that have appeared
so far. We will use the mapping $T^{0}(\cdot,\cdot,\cdot)$ below but $T^{n}(\cdot,\cdot,\cdot)$ 
for any $n$ would do equally well. 
Notice that
$$
\mathcal{L}(T^{0}(\mathbf{U}^{*},\tilde{X}_{M_{m}},Y_{0}),(Y_{1+k})_{k\in\mathbb{Z}})
=\mathcal{L}(\tilde{X}_{M_{m}+1},(Y_{k})_{k\in\mathbb{Z}})
$$
and then from \eqref{fors}, necessarily,
$$
||\mathcal{L}(T^{0}(\mathbf{U}^{*},\tilde{X}_{M_{m}},Y_{0}),(Y_{1+k})_{k\in\mathbb{Z}})
-\mu_{\sharp}||_{TV}\leq 2\varepsilon.
$$
Now employing the limiting random variable representing $\mu_{\sharp}$,
\begin{eqnarray*} & &
||\mathcal{L}(T^{0}(\mathbf{U}^{*},\tilde{X}_{M_{m}},Y_{0}),(Y_{1+k})_{k\in\mathbb{Z}})
-\mathcal{L}(T^{0}(\mathbf{U}^{*},\tilde{X}_{\infty},Y_{0}),(Y_{1+k})_{k\in\mathbb{Z}}))||_{TV}\\
&\leq& 2P(T^{0}(\mathbf{U}^{*},\tilde{X}_{M_{m}},Y_{0})\neq T^{0}(\mathbf{U}^{*},\tilde{X}_{\infty},Y_{0}))\\
&\leq& 2P(\tilde{X}_{M_{m}}\neq \tilde{X}_{\infty})\leq 2\varepsilon.
\end{eqnarray*}
Thus we have
$$
||\mathcal{L}(T^{0}(\mathbf{U}^{*},\tilde{X}_{\infty},Y_{0}),(Y_{1+k})_{k\in\mathbb{Z}})-\mu_{\sharp}||_{TV}\leq 4\varepsilon
$$
and, as $\varepsilon$ was arbitrary, 
$\mathcal{L}(T^{0}(\mathbf{U}^{*},\tilde{X}_{\infty},Y_{0}),(Y_{1+k})_{k\in\mathbb{Z}})=\mu_{\sharp}$
follows. Clearly, this means that $$
\mathcal{L}(X_{0},(Y_{k})_{k\in\mathbb{Z}})=\mu_{\sharp}\mbox{ implies }
\mathcal{L}(X_{1},(Y_{1+k})_{k\in\mathbb{Z}})=\mu_{\sharp}$$ 
and the latter
extends immediately to $\mathcal{L}(X_{t},(Y_{t+k})_{k\in\mathbb{Z}})=\mu_{\sharp}$ for all $t\geq 2$, too. 
The proof is complete.
\end{proof}

  Before transitioning to the analysis of continuous-time processes,
  let us demonstrate the application of Theorem \ref{maine1} on a
  benchmark model: the discrete-time counterpart of \eqref{diffi1} with log-Gaussian $V_{t}$.
We take the simplest mean-reverting drift, but the same argument applies under
more general dissipativity conditions. 

%  and with the simplest mean-reverting drift. Even in this simpler, discrete-time
%  setting proving stochastic stability appears to be rather challenging and does not
%  follow from previous methods.
  
  \begin{example}
{\rm    Consider the following 
model for financial time
    series. Let $\eta_t,t\in\mZ$ be independent standard Gaussian random
    variables and
    $$
    Z_t = \sum_{k=0}^\infty a_k \eta_{t-k},
    $$
    a causal moving average with constants $a_k$, $k\in \mN$ 
    satisfying $\sum_{k}
    a_k^2<\infty$. Therefore $Z_t$ is almost surely well defined and
    is a stationary Gaussian process. $Z_t$ represents the
    log-volatility of an asset's log-price $X_t$ which in turn is defined as
    \begin{equation}
      \label{eq:disc-logvol}
      X_{t+1} = \gamma X_t + \rho e^{Z_t} \eta_{t+1} +
      \sqrt{1-\rho^2} e^{Z_t} \eps_{t+1},
    \end{equation}
    where $\gamma\in (0,1), \rho\in (-1,1)$ and $\eps_k,k\in\mN$ is an i.i.d.\
    series of zero-mean random variables, also independent of $\eta_t,t\in\mZ$.{}
 %   One could equally consider multi-dimensional $X$ with 
 % $\nu(X_{t})$ replacing $\gamma X_{t}$ in \eqref{eq:disc-logvol} with some
 % dissipative function $\nu$. To keep details simple we stay 
 % in this setting.
  
    For the $\eps_k$ we assume they have finite variance and have a
    positive density $f(x)$ such that for all $n\in\mN$,
    $\inf_{x\in[-n,n]}f(x)=c(n)>0$.
    Additionally, we assume the initial price $X_0$ has finite
    variance and is independent of
    $\eta_t,t\in\mZ, \eps_k,k\in\mN$.

  We claim that under these natural assumptions Theorem \ref{maine1} is applicable to
  the model \eqref{eq:disc-logvol}. 
  
  First of all, the random environment is defined as
  $Y_{t}:=(Z_{t},\eta_{t+1})$. %The small sets are chosen as follows for the
  %target process and the random environment for $n\in\mN$:
  We choose
  $$\mathcal{X}_n = \{x\in\mR :|x|\le n\} \qquad \mathcal{Y}_n = \{(z,\eta)\in\mR^2 :|z|,|\eta|\le n\}.$$

  We first verify Assumption \ref{minor1}, fix some $n\in\mN$ and
  $\nu_n=\frac 1 2 Leb\vert_{[-1,1]}$. Now that we are working with
  absolutely continuous distributions, we have to find a lower bound
  of the transition density to $[-1,1]$ from any departure point
  $X_t\in\mathcal{X}_n, (Z_t,\eta_{t+1})\in\mathcal{Y}_{n}$.

  Rearranging \eqref{eq:disc-logvol}, we get
  $$
  \eps_{t+1} = \frac{X_{t+1}-\gamma X_t}{\sqrt{1-\rho^2}e^{Z_t}} +
  \frac{\rho}{\sqrt{1-\rho^2}}\eta_{t+1}.
  $$
  Requiring $X_{t+1}$ to arrive in $[-1,1]$, knowing
  $X_t,\eta_{t+1}\in [-n,n]$, $e^{Z_t}\in [e^{-n},e^n]$, the possible
  needed values of $\eps_{t+1}$ are restricted within some bounded
  interval $[-d(n),d(n)]$. Using the condition on the bounded
  positivity of the density $f(x)$ of $\eps_{t+1}$ we get a valid
  minorization with
  $$
  \alpha_n = 2\inf_{x\in[-d(n),d(n)]}f(x) = 2c(d(n))>0.
  $$

  It is left to confirm that $X_0\in\mathcal{M}_b$, so that $X_t$
  uniformly rarely leaves the small sets. By recursively using
  \eqref{eq:disc-logvol} we may express $X_t$ as follows:
  \begin{equation}
    \label{eq:logvol-expand}
    X_t = \sum_{s=1}^t \gamma^{t-s} e^{Z_{s-1}}\left(\rho\eta_s +
    \sqrt{1-\rho^2}\eps_s\right) + \gamma^t X_0.
  \end{equation}
  To bound $X_t$, we compute $E[X^2_t]$. Observe that when evaluating the
  square of this sum, all cross-terms cancel when taking expectation,
  even the ones only involving $Z$ and $\eta$.
 Consequently,
  $$
  E[X_t^2] = \sum_{s=1}^t \gamma^{2t-2s} E\left[e^{2Z_{s-1}}\right](\rho^2 E[\eta^2_s] +
  (1-\rho^2)E[\eps^2_s]) + \gamma^{2t} E[X^2_0].
  $$
  Regarding these terms, remember that $Z_{t}$ was Gaussian thus it
  has finite exponential moments and all appearing variables had finite
  variances. 
  Moreover, due to the stationarity of all components
  appearing, 
  we have the time-independent bound
  $$
  E[X_t^2] \le \frac{1}{1-\gamma^2} E\left[e^{2Z_{0}}\right](\rho^2 E[\eta^2_1] +
  (1-\rho^2)E[\eps^2_1]) + E[X^2_0] =: K <\infty.
  $$
  From here we can conveniently bound
  $$
  \sup_{t\in\mN}P(|X_t|>n)\le \frac {K}{n^2},
  $$
  which indeed converges to $0$ as $n\to\infty$.
  This reasoning shows that 
  $\mathcal{L}(X_0,(Z_k,\eta_{k+1})_{k\in\mZ})\in\mathcal{M}_b$.
  We have verified the minorization Assumption
  \ref{minor1} just before so Theorem \ref{maine1} applies, ensuring 
convergence 
in total variation. The present example complements 
Example 3.4 of \cite{balazs} where convergence in
total variation was established under stronger assumptions (but with a rate estimate).}
\end{example}

\section{Proofs in continuous time}\label{stabi}

Until finishing the proof of Theorem \ref{stability}, we assume that all the hypotheses 
of that theorem are in force.
Let us first establish a simple continuity property.

\begin{lemma}\label{ucont} When $s\to 0$, $\sup_{t\in\mathbb{R}}
\{E[\mathbf{d}_{d\times d}(\mathbf{V}_{t+s},\mathbf{V}_{t})]+E[\mathbf{d}_{d\times m}(\mathbf{R}_{t+s},\mathbf{R}_{t})]\}\to 0$
holds true.	
\end{lemma}
\begin{proof} By stationarity of $V$, $\rho$, this amounts to proving 
$E[\mathbf{d}_{d\times d}(\mathbf{V}_{s},\mathbf{V}_{0})]+E[\mathbf{d}_{d\times m}(\mathbf{R}_{s},\mathbf{R}_{0})]\to 0$.
The process $V$ has trajectories that are uniformly continuous on compacts
hence $\sup_{u\in [i,i+1]}|V_{u+s}-V_{u}|\to 0$ almost surely as $s\to 0$, 
for each $i\in\mathbb{Z}$. Then
$E[1\wedge\sup_{u\in [i,i+1]}|V_{u+s}-V_{u}|]\to 0$ for each $i$ 
and, finally, $E[\mathbf{d}_{d\times d}(\mathbf{V}_{s},\mathbf{V}_{0})]\to 0$
by the definition of $\mathbf{d}$. We argue in the same manner for $\mathbf{R}$. 
\end{proof}

Now let us prove a moment estimate.

\begin{lemma}\label{lyapunov}
We have $\tilde{L}:=\sup_{t\in\mathbb{R}_{+}}E[|L_{t}|^{2}]<\infty$.
\end{lemma}
\begin{proof} Fix $k\in\mathbb{N}$, for the moment.
Define the stopping times $\tau_l:=\inf\{ t>k:|L_t|>l\}$ for $l\in\mathbb{N}$. 
It\^{o}'s formula and Assumption \ref{dissipi} imply that, for all $k\leq t\leq k+1$, 	
\begin{eqnarray*}
\mathrm{e}^{2\alpha(t\wedge\tau_{l}-k)}|L_{t\wedge\tau_{l}}|^{2} &=&
|L_{k}|^{2} + \int_{k}^{t\wedge\tau_{l}}
2\mathrm{e}^{2\alpha(s-k)}\langle L_{s},\zeta(L_{s},V_{s})\rangle\, ds
+ \int_{k}^{t\wedge\tau_{l}} 2\mathrm{e}^{2\alpha(s-k)} L_{s}^{*}V_{s}d\overline{W}_{s}\\
&+& \int_{k}^{t\wedge\tau_{l}} \mathrm{e}^{2\alpha(s-k)}\mathrm{tr}(V_{s}^{*}V_{s})\, ds+ 
\int_{k}^{t\wedge\tau_{l}}2\alpha 
\mathrm{e}^{2\alpha(s-k)}|L_{s}|^{2}\, ds \\
&\leq& |L_{k}|^{2}-\int_{k}^{t\wedge\tau_{l}}2\alpha\mathrm{e}^{2\alpha(s-k)}|L_{s}|^{2}\, ds
+ \int_{k}^{k+1}2\mathrm{e}^{2\alpha(s-k)}\beta(1+|V_{s}|^{\xi})\, ds \\
&+& \int_{k}^{t\wedge\tau_{l}} 2\mathrm{e}^{2\alpha(s-k)} 
L_{s}^{*}V_{s}d\overline{W}_{s} + \mathrm{e}^{2\alpha}\int_{k}^{k+1}d|V_{s}|^{2}\, ds\\
&+& \int_{k}^{t\wedge\tau_{l}}2\alpha 
\mathrm{e}^{2\alpha(s-k)}|L_{s}|^{2}\, ds\\
&\leq& |L_{k}|^{2}+\int_{k}^{k+1}2\mathrm{e}^{2\alpha}\beta(1+|V_{s}|^{\xi})\, ds \\
&+& \int_{k}^{t\wedge\tau_{l}} 2\mathrm{e}^{2\alpha(s-k)} 
L_{s}^{*}V_{s}d\overline{W}_{s} + \mathrm{e}^{2\alpha}\int_{k}^{k+1}d|V_{s}|^{2}\, ds,
\end{eqnarray*}
where $\overline{W}_{t}=\int_{0}^{t}\rho_{s}\,dB_{s}+\int_{0}^{t}\sqrt{I-\rho_{s}\rho^{*}_{s}}\, dW_{s}$
is a $d$-dimensional standard Brownian motion. Taking expectations and noting
the martingale property of the stochastic integral,
$$
E[\mathrm{e}^{2\alpha(t\wedge\tau_{l}-k)}|L_{t\wedge\tau_{l}}|^{2}]\leq 
E[|L_{k}|^{2}]+\int_{k}^{k+1}\mathrm{e}^{2\alpha}(2\beta+d)(1+E[|V_{s}|^{\xi}])\, ds. 
$$
Noting stationarity of $V$ and applying Fatou's lemma,
\begin{equation}\label{selberg}
E[|L_{t}|^{2}]\leq 
\mathrm{e}^{-2\alpha (t-k)}E[|L_{k}|^{2}]+e^{2\alpha}(2\beta+d)(1+E[|V_{0}|^{\xi}]). 
\end{equation}
Setting $t=k+1$, Assumption \ref{initi} and an induction on $k$ show 
that $E[|L_{k}|^{2}]<\infty$ for all $k\in\mathbb{N}$,
in fact, $\sup_{k}E[|L_{k}|^{2}]<\infty$. Finally,
also $\sup_{t\geq 0}E[|L_{t}|^{2}]<\infty$ follows from \eqref{selberg}.{}
\end{proof}

Let $\mathbf{C}_{k}$ denote the Banach space of continuous $\mathbb{R}^{k}$-valued functions 
on $[0,1]$ equipped with the usual
maximum norm $||\cdot||_{\mathbf{C}_{k}}$. 
The family of functions in $\mathbf{C}_{d\times d}$ whose values are non-singular
is denoted by $\mathbf{C}^{+}$. We further define 
$$
\mathbf{C}^{1}:=\{\mathbf{r}\in\mathbf{C}_{d\times m}:\, \mathbf{r}_{t}\mathbf{r}_{t}^{*}\leq I,\ t\in [0,1]\}
$$
as well as
$$
\mathbf{C}^{1+}:=\{\mathbf{r}\in\mathbf{C}_{d\times m}:\, \mathbf{r}_{t}\mathbf{r}^{*}_{t}< I,\ t\in [0,1]\}.
$$

The auxiliary process to be defined in \eqref{aix} below plays a key role in our arguments.
The parameters $\mathbf{v},\mathbf{r}$ represent the ``frozen'' values of trajectories of the
volatility and correlation processes, while $\mathbf{z}$ will be a generic value of
the stochastic integral of $V\rho$ with respect to $B$.

For each $\mathbf{v}\in\mathbf{C}_{d\times d}$, $\mathbf{z}\in\mathbf{C}_{d}$, 
$\mathbf{r}\in\mathbf{C}^{1}$ and $x\in\mathbb{R}^{d}$, let 
$\tilde{X}_t(\mathbf{v},\mathbf{z},\mathbf{r},x)$, $t\in [0,1]$ denote the unique $\mathcal{F}_t$-adapted
solution of the SDE
\begin{equation}\label{aix}
d\tilde{X}_t(\mathbf{v},\mathbf{z},\mathbf{r},x)=\zeta\left(\tilde{X}_t(\mathbf{v},\mathbf{z},\mathbf{r},x)
+ \mathbf{z}_t,\mathbf{v}_{t}\right)\, dt+
\mathbf{v}_t\sqrt{I-\mathbf{r}_t\mathbf{r}_t^{*}}\, dW_t,\ \tilde{X}_0(\mathbf{v},\mathbf{z},\mathbf{r},x)=x,
\end{equation}
which exists e.g.\ by Theorem 7 on page 82 of \cite{krylov}.
We shall use the shorthand notation $\mathbf{q}:=(\mathbf{v},\mathbf{z},\mathbf{r},x)$ in the sequel.
Introduce also the space $\mathcal{Y}:=\mathbf{C}_{d\times d}\times\mathbf{C}_{d}\times\mathbf{C}^{1}$ where
the random environment (to be defined in \eqref{schubert} below) will evolve.

In line with the notations of the standard reference work \cite{nualart},
$\mathbb{D}^{k,p}$ denotes the $p$-Sobolev space of $k$ times Malliavin differentiable functionals.
The first and second Malliavin derivative of a functional $F$ will be denoted by $DF$, $D^{2}F$ or
$D_{r}F$, $D^{2}_{r_{1},r_{2}}F$ when we need to emphasize that these are random processes/fields
indexed by $r,r_{1},r_{2}$.
The Skorokhod integral operator (the adjoint of $D$) is denoted by $\delta$.
The notation $H$ refers to the Hilbert-space of square-integrable $\mathbb{R}^{d}$-valued
functions on $[0,1]$.

For $F=(F^{1},\ldots,F^{d})$ with $F^{i}\in \mathbb{D}^{1,2}$, $i=1,\ldots,d$ 
the corresponding Malliavin matrix $\sigma(F)$ is defined
as
$$
\sigma(F)_{ij}=\sum_{l=1}^{d}\int_{0}^{1} D^{(l)}_{s}F^{i} D^{(l)}_{s}F^{j}\, ds,
$$
where $D^{(l)}F^{i}$ denotes the $l$th coordinate of $DF^{i}$.
In the sequel, the notation $L^p$ refers to the usual space of 
$p$-integrable real-valued random variables, for $p\geq 1$. We define $\gamma:=\sigma^{-1}$
on the event where $\sigma$ is invertible and $0$ otherwise.

\begin{lemma}\label{conditions} For each 
$\mathbf{q}\in\mathcal{Y}\times\mathbb{R}^{d}$, we have 
$\tilde{X}_1(\mathbf{q})\in \cap_{p\geq 1}\mathbb{D}^{2,p}$, $D\tilde{X}_1(\mathbf{q})$
and $D^{2}\tilde{X}_1(\mathbf{q})$ are bounded.
Furthermore, if $\mathbf{v}\in\mathbf{C}^{+}$ and 
$\mathbf{r}\in \mathbf{C}^{1+}$ then 
$\gamma$ is uniformly bounded,
in particular, $1/\mathrm{det}(\sigma(\tilde{X}_1(\mathbf{q})))\in \cap_{p\geq 1}L^p$ holds.
\end{lemma}
\begin{proof}
The first statement follows from the proof of Theorem 2.2.2 of \cite{nualart} which applies in the cases 
$N=1,2$ by Assumption \ref{thrice}.

To see the last statement, recall from Theorem 2.2.1 of
\cite{nualart} that the matrix-valued process $(M_{t}(u))_{ij}:=D^{(j)}\tilde{X}^{i}_{t}(\mathbf{q})$, 
$t\in [u,1]$ satisfies the (random) ordinary differential equation 
\begin{equation}\label{mm}
dM_{t}(u)=\partial_{x}\zeta\left(\tilde{X}_t(\mathbf{q})+\mathbf{z}_{t},\mathbf{v}_{t}\right)
M_{t}(u)\, dt,\ M_{u}(u)=\mathbf{v}_u\sqrt{I-\mathbf{r}_{u}\mathbf{r}_{u}^{*}},
\end{equation}
for each $u\in [0,1]$. Assumption \eqref{thrice} gives us
$K'$ with  $|\partial_{x}\zeta|\leq K'$. We see 
that \begin{eqnarray*} 
M_{1}(u)M_{1}^{*}(u) &=& \mathbf{v}_u\sqrt{I-\mathbf{r}_{u}\mathbf{r}_{u}^{*}}
\exp\left({}
\int_{u}^{1}\partial_{x}\zeta(\tilde{X}_s(\mathbf{q})+\mathbf{z}_{s},\mathbf{v}_{s})
+\partial_{x}\zeta^{*}(\tilde{X}_s(\mathbf{q})+\mathbf{z}_{s},\mathbf{v}_{s})\, ds 
\right)\sqrt{I-\mathbf{r}_{u}\mathbf{r}_{u}^{*}}\mathbf{v}_u^{*}\\
&\geq& \exp(-2K')\mathbf{v}_u(I-\mathbf{r}_{u}\mathbf{r}_{u}^{*})\mathbf{v}_u^{*}\\
&\geq& \epsilon I,
\end{eqnarray*}
for some $\epsilon=\epsilon(\mathbf{v},\mathbf{r})>0$ since
$\mathbf{v}\in\mathbf{C}^{+}$ and $\mathbf{r}\in \mathbf{C}^{1+}$.
This implies our claim since $$
\sigma(\tilde{X}_s(\mathbf{q}))=\int_{0}^{1} M_{1}(u)M_{1}(u)^{*}\, du.
$$

It follows also from \eqref{mm} that $D\tilde{X}_1(\mathbf{q})$ is bounded, for each $\mathbf{q}$.
Finally, Lemma 2.2.2 of \cite{nualart} implies that, for each pair of indices $i,l$, the second derivative
$DD^{(l)}\tilde{X}^{i}_{t}(\mathbf{q})$ satisfies for all indices $j$ and for all $s\leq u\leq t$,
\begin{eqnarray*}
|D_{s}^{(j)}D_{u}^{(l)}\tilde{X}_{t}^{i}(\mathbf{q})| &=& \left|
\int_{u}^{t}(D_{s}^{(j)}\partial_{x}\zeta(\tilde{X}_{r}(\mathbf{q})+\mathbf{z}_{r},\mathbf{v}_{r}))
D_{u}^{(l)}\tilde{X}^{i}_{r}(\mathbf{q})+\partial_{x}
\zeta(\tilde{X}_{r}(\mathbf{q})+\mathbf{z}_{r},\mathbf{v}_{r})
D_{s}^{(j)}D_{u}^{(l)}\tilde{X}_{r}^{i}(\mathbf{q})\, dr\right|\\
&\leq& 
\int_{u}^{t} K_{1}+K_{2}
|D_{s}^{(j)}D_{u}^{(l)}\tilde{X}_{r}^{i}(\mathbf{q})|\, dr
\end{eqnarray*}
for constants $K_{1},K_{2}$ since $\partial_{x}\zeta$, $\partial_{xx}\zeta$, $D\tilde{X}^{i}_{t}(\mathbf{q})$
are all bounded. Gronwall's lemma now guarantees that also
$D_{s}^{(j)}D_{u}^{(l)}\tilde{X}_{1}^{i}(\mathbf{q})$ is bounded. 
\end{proof}

We now set up a discrete-time machinery so that we can invoke the results of Subsection \ref{nonmarki}.
Set $\mathcal{X}:=\mathbb{R}^{d}$ and $\mathcal{X}_n:=\{x\in\mathbb{R}^{d}:\ |x|\leq n\}$, $n\in\mathbb{N}$.
Define, for $k\in\mathbb{Z}$, the $\mathcal{Y}$-valued random variables
\begin{equation}\label{schubert}
Y_k:=\left( (V_{k+t})_{t\in [0,1]},(Z_{k+t}-Z_{k})_{t\in [0,1]},(\rho_{k+t})_{t\in [0,1]}\right), 
\end{equation}
where we denote $Z_u:=\int_0^u V_{s}\rho_s\, dB_s$, $u\in\mathbb{R}_{+}$. $Y$ is a stationary
process by Assumption \ref{stationary}.

By Prokhorov's theorem, there exist an increasing sequence of compact sets 
$\mathbf{D}_n\subset \mathbf{C}_{d\times d}\times\mathbf{C}_{d}\times\mathbf{C}^{1}$, $n\in\mathbb{N}$
such that $P(Y_0\notin\mathbf{D}_n)\leq 1/n$. 
As $V\in\mathcal{V}$ and $\rho\in \mathcal{R}$, 
$P(Y_{0}\in \mathbf{C}^{+}\times\mathbf{C}_{d}\times\mathbf{C}^{1+})=1$ holds.
Thus there is an increasing $\mathbb{N}$-valued
sequence $l(n)\to\infty$, $n\to\infty$ such that the sets
$$
\mathcal{Y}_n:=\{(\mathbf{v},\mathbf{z},\mathbf{r})\in\mathbf{D}_{n}:\ \mathbf{r}_{s}\mathbf{r}_{s}^{*}\leq 
(1-1/l(n))I,\ \mathbf{v}_{s}\mathbf{v}^{*}_{s}\geq I/l(n),\ s\in [0,1] \}
$$
satisfy $P(Y_{0}\notin\mathcal{Y}_{n})\leq 2/n$, $n\in\mathbb{N}$. 
Being closed subsets of the respective $\mathbf{D}_{n}$, they are compact and
satisfy $P(Y_{0}\notin\mathcal{Y}_{n})\to 0$, $n\to\infty$.

We define a metric on $\mathcal{Q}:=\mathcal{Y}\times\mathbb{R}^{d}$ by setting, 
for $\mathbf{q}_{i}=(\mathbf{v}^{i},
\mathbf{z}^{i},\mathbf{r}^{i},x^{i})$, $i=1,2$,
$$
\rho(\mathbf{q}_{1},\mathbf{q}_{2}):=||\mathbf{v}^{1}-\mathbf{v}^{2}||_{\mathbf{C}_{d\times d}}
+||\mathbf{z}^{1}-\mathbf{z}^{2}||_{\mathbf{C}_{d}}+||\mathbf{r}^{1}-\mathbf{r}^{2}||_{\mathbf{C}_{d\times m}}
+|x^{1}-x^{2}|.
$$
Continuity of $\tilde{X}_t(\mathbf{q})$ and its Malliavin derivatives with respect to the parameter
$\mathbf{q}$ is established next.

\begin{lemma}\label{contu} For each $n\in\mathbb{N}$ and $p\geq 2$ there
exists $C(n,p)>0$ such that for all $\mathbf{q}^{1},\mathbf{q}^{2}\in \mathcal{Y}_{n}\times\mathcal{X}_{n}$
we have
\begin{eqnarray*}
E^{1/p}[\sup_{t\in [0,1]}|\tilde{X}_t(\mathbf{q}^{1})-\tilde{X}_t(\mathbf{q}^{2})|^{p}] &\leq& C(n,p){}
\rho(\mathbf{q}^{1},
\mathbf{q}^{2}),\\
E^{1/p}\left[ ||D\tilde{X}_1(\mathbf{q}^{1})- D\tilde{X}_1(\mathbf{q}^{2})||_{H}^{p}\right]
&\leq& C(n,p)\rho(\mathbf{q}^{1},
\mathbf{q}^{2}),
\\
E^{1/p}\left[||D^2\tilde{X}_1(\mathbf{q}^{1})-D^2\tilde{X}_1(\mathbf{q}^{2})||_{H\otimes H}^{p}\right]
&\leq& C(n,p)\rho(\mathbf{q}^{1},
\mathbf{q}^{2}).
\end{eqnarray*}
\end{lemma}
\begin{proof} For $i=1,2$, define the Picard iterates $Z^{i}_{0}(t):=x^{i}$, $t\in [0,1]$ and
\begin{equation}
Z^{i}_{l+1}(t):=x^{i}+\int_{0}^{t}\zeta\left(Z^{i}_{l}(s)
+\mathbf{z}^{i}_s,\mathbf{v}^{i}_{s}\right)\, ds+
\int_{0}^{t}\mathbf{v}^{i}_s\sqrt{I-\mathbf{r}^{i}_s\mathbf{r}^{i*}_{s}}\, dW_s,
\end{equation}
for $t\in[0,1]$ and $l\in\mathbb{N}$. Let $K_{0}$ denote a bound for $|\partial_{x}\zeta|$ such that 
$K_{0}\geq K$ where $K$ is as in Assumption \ref{thrice}. Clearly,
\begin{eqnarray*}
& & \sup_{u\in [0,t]}|Z^{1}_{l+1}(u)-Z^{2}_{l+1}(u)| \\ 
&\leq& |x^{1}-x^{2}| + K_{0}\int_{0}^{t}\sup_{u\in [0,s]}|Z^{1}_{l+1}(u)-
Z^{2}_{l+1}(u)|+(1+|\mathbf{v}^{1}_{s}|+|\mathbf{z}^{1}_{s}-\mathbf{z}^{2}_{s}|+
|\mathbf{v}^{2}_{s}|)|\mathbf{v}^{1}_{s}-\mathbf{v}^{2}_{s}|\, ds\\
&+& \sup_{s\in [0,t]}\left|\int_{0}^{s} \mathbf{v}^{1}_{u}\sqrt{I-\mathbf{r}^{1}_{u}\mathbf{r}^{1*}_{u}}
-\mathbf{v}^{2}_{u}\sqrt{I-\mathbf{r}^{2}_{u}\mathbf{r}^{2*}_{u}}dW_{u}\right|.
\end{eqnarray*}
Note that there is $C>0$ such that $$
(x_{1}+x_{2}+x_{3}+x_{4}+x_{5})^{2}\leq C(x^{2}_{1}+x_{2}^{2}+x_{3}^{2}+x_{4}^{2}+x_{5}^{2}),\ x_{i}\in\mathbb{R},\ i=1,2,3,4,5.
$$
Taking squares and using Cauchy's inequality we arrive at
\begin{eqnarray*}
& & \sup_{u\in [0,t]}|Z^{1}_{l+1}(u)-Z^{2}_{l+1}(u)|^{2} \\
&\leq& C|x^{1}-x^{2}|^{2}+CtK_{0}^{2}\int_{0}^{t}
\sup_{u\in [0,s]}|Z^{1}_{l+1}(u)-
Z^{2}_{l+1}(u)|^{2}+
|\mathbf{z}^{1}_{s}-\mathbf{z}^{2}_{s}|^{2}+
(1+|\mathbf{v}^{1}_{s}|+|\mathbf{v}^{2}_{s}|)^{2}
|\mathbf{v}^{1}_{s}-\mathbf{v}^{2}_{s}|^{2}\, ds\\
&+& C\sup_{s\in [0,t]}\left|\int_{0}^{s} \mathbf{v}^{1}_{u}\sqrt{I-\mathbf{r}^{1}_{u}\mathbf{r}^{1*}_{u}}
-\mathbf{v}^{2}_{u}\sqrt{I-\mathbf{r}^{2}_{u}\mathbf{r}^{2*}_{u}}dW_{u}\right|^{2}.
\end{eqnarray*}
Taking expectations, applying Doob's inequality and noting $t\leq 1$,
\begin{eqnarray*}
& & E\sup_{u\in [0,t]}|Z^{1}_{l+1}(u)-Z^{2}_{l+1}(u)|^{2} \\
&\leq& C|x^{1}-x^{2}|^{2}+CK_{0}^{2}\int_{0}^{t}
E\sup_{u\in [0,s]}|Z^{1}_{l+1}(u)-
Z^{2}_{l+1}(u)|^{2}+(1+|\mathbf{v}^{1}_{s}|+
|\mathbf{z}^{1}_{s}-\mathbf{z}^{2}_{s}|^{2}+|\mathbf{v}^{2}_{s}|)^{2}
|\mathbf{v}^{1}_{s}-\mathbf{v}^{2}_{s}|^{2}\, ds\\
&+& 4CE\int_{0}^{1} \left(\mathbf{v}^{1}_{u}\sqrt{I-\mathbf{r}^{1}_{u}\mathbf{r}^{1*}_{u}}
-\mathbf{v}^{2}_{u}\sqrt{I-\mathbf{r}^{2}_{u}\mathbf{r}^{2*}_{u}}\right)^{2}\, du\\
&\leq& C_{n}'\left[\rho^{2}(\mathbf{q}^{1},\mathbf{q}^{2})+\int_{0}^{t}E\sup_{u\in [0,s]}|Z^{1}_{l+1}(u)-
Z^{2}_{l+1}(u)|^{2}\, ds\right].
\end{eqnarray*}
for suitable $C_{n}'$ because
$z\to \sqrt{I-zz^{*}}$ is Lipschitz-continuous on the set $\{z\in\mathbf{C}^{1}:
zz^{*}\leq (1-\epsilon)I\}$, for all $\epsilon>0$.{}
Gr\"onwall's lemma implies that for some constant $C_{n}''$, independent of $l$,
\begin{eqnarray*}
E\left[\sup_{t\in [0,1]}|Z^{1}_{l+1}(t)-Z^{2}_{l+1}(t)|^{2}\right]\leq C_{n}''
\rho^{2}(\mathbf{q}^{1},
\mathbf{q}^{2}).
\end{eqnarray*}
Since Picard iterates converge, (see e.g.\ Lemma 2.2.1 in \cite{nualart}), we get
$$
E^{1/2}\left[\sup_{t\in [0,1]}|\tilde{X}(\mathbf{q}^{1})-\tilde{X}(\mathbf{q}^{2})|^{2}\right]\leq 
\sqrt{C_{n}''}\rho(\mathbf{q}^{1},
\mathbf{q}^{2}).
$$
A similar argument works in $L^{p}$ with $p> 2$, too.
Now recall that $D\tilde{X}(\mathbf{q})$, $D^{2}\tilde{X}(\mathbf{q})$ also satisfy similar (even simpler) 
equations, see Theorem 2.2.1 of \cite{nualart}, so analogous arguments apply to them, proving the
remaining two inequalities.
\end{proof}

We continue with some more technical material. In the following lemma, we will rely on the
powerful techniques presented in \cite{bally-notes}.

\begin{lemma}\label{compact} The random variables $\tilde{X}_{1}(\mathbf{q})$ have
densities $p_{\mathbf{q}}(u)$, $u\in\mathbb{R}^{d}$ with respect to $\mathrm{Leb}_{d}$,
the $d$-dimensional Lebesgue measure, for each $\mathbf{q}\in\mathcal{Y}_{n}\times\mathcal{X}_{n}$,
for each $n$. These densities have versions such that the mapping $(u,\mathbf{q})\to p_{\mathbf{q}}(u)$
is continuous on $\mathcal{Y}_{n}\times\mathcal{X}_{n}$, for each $n$.
\end{lemma}
\begin{proof} Fix $n$, let $\mathbf{q}_{k}\in\mathcal{Y}_{n}\times\mathcal{X}_{n}$
and $u_{k}\in\mathbb{R}^{d}$, $k\in\mathbb{N}$ be such that
$$
(u_n,\mathbf{q}_n):=(u_n,\mathbf{v}_n,\mathbf{r}_n,\mathbf{z}_n,x_n)\to
(u,\mathbf{q}):=(u,\mathbf{v},\mathbf{r},\mathbf{z},x)
$$ 
hold as $n\to\infty$.
{}
Let $\partial_i Q$, $i=1,\ldots,d$ denote the $i$th partial derivative of the 
Poisson kernel on $\mathbb{R}^{d}$, see page 14 of \cite{bally-notes}.
We rely on the Malliavin-Thalmaier formula for the density of
functionals on the Wiener space \cite{malliavin-thalmaier}, as presented in \cite{bally-notes}.
By Theorem 2.3.1 of \cite{bally-notes}, the representation
\begin{equation}\label{suruseg}
p_{\mathbf{q}}(u)=\sum_{i=1}^{d}\sum_{j=1}^{d}E\left[\partial_i Q\left(\tilde{X}_{1}(\mathbf{q})-u\right){}
\delta\left(\gamma^{i,j}(\tilde{X}_{1}(\mathbf{q}))D\tilde{X}_{1}(\mathbf{q})\right)
\right]	
\end{equation}
provides the density function of $\tilde{X}_{1}(\mathbf{q})$ (with respect to the $d$-dimensional
Lebesgue-measure). 

Fix $i,j$. By Lemma \ref{conditions}, the sequence $\tilde{X}_{1}(\mathbf{q}_{n})$
is bounded in $\mathbb{D}^{2,p}$ for all $p$ and $\gamma(\tilde{X}_{1}(\mathbf{q}_{n}))$ is uniformly
bounded, hence $\partial_i Q\left(\tilde{X}_{1}(\mathbf{q}_{n})-u\right)$ is bounded in $L^{(d+1)/d}$
by (2.86) in Theorem 2.3.1 of \cite{bally-notes}. By Corollary 2.2.12 in \cite{bally-notes},
the sequence $\delta\left(\gamma^{i,j}(\tilde{X}_{1}(\mathbf{q}))D\tilde{X}_{1}(\mathbf{q}_{n})\right)${}
is bounded in $L^{p}$ for all $p\geq 1$, in particular in $L^{(d+1/2)(2d+2)/d}$. But then Hölder's
inequality implies 
\begin{eqnarray}\nonumber
& & \sup_{n}E\left[\left|\partial_i Q\left(\tilde{X}_{1}(\mathbf{q}_{n})-u\right){}
\delta\left(\gamma^{i,j}(\tilde{X}_{1}(\mathbf{q}_{n}))D\tilde{X}_{1}(\mathbf{q}_{n})\right)\right|^{(d+1/2)/d}
\right]\\
&\leq& \sup_{n}E\left[\left|\partial_i Q\left(\tilde{X}_{1}(\mathbf{q}_{n})-u\right)\right|^{(d+1)/d}\right]^{(d+1/2)/(d+1)}\\
&\times&
\sup_{n}E\left[\left|\delta^{(d+1/2)(2d+2)/d}\left(\gamma^{i,j}(\tilde{X}_{1}(\mathbf{q}_{n}))
D\tilde{X}_{1}(\mathbf{q}_{n})\right)\right|\right]^{1/(2d+2)}<\infty.\label{fgh}
\end{eqnarray}
This means that the sequence $\partial_i Q\left(\tilde{X}_{1}(\mathbf{q}_{n})-u\right){}
\delta\left(\gamma^{i,j}(\tilde{X}_{1}(\mathbf{q}_{n}))D\tilde{X}_{1}(\mathbf{q}_{n})\right)$, $n\in\mathbb{N}$
is uniformly integrable
hence it suffices to prove that
\begin{equation}\label{concon}
\partial_i Q\left(\tilde{X}_{1}(\mathbf{q}_{n})-u_{n}\right){}
\delta\left(\gamma^{i,j}(\tilde{X}_{1}(\mathbf{q}_{n}))D\tilde{X}_{1}(\mathbf{q}_{n})\right)\to{}
\partial_i Q\left(\tilde{X}_{1}(\mathbf{q})-u\right){}
\delta\left(\gamma^{i,j}(\tilde{X}_{1}(\mathbf{q}))D\tilde{X}_{1}(\mathbf{q})\right),\ n\to\infty
\end{equation}
in probability. We remark that in \eqref{fgh} we can also take a supremum in $u\in C$ for any
compact $C\subset\mathbb{R}^{d}$ and this also implies
\begin{equation}\label{morso}
\sup_{n}\sup_{u\in C}p_{\mathbf{q}_{n}}(u)<\infty.	
\end{equation}

We start by treating the factor $\partial_i Q$ in \eqref{concon}. Let $\varepsilon>0$ be given. 
Notice that each $\partial_i Q$
is Lipschitz-continuous outside a ball of radius $\varepsilon$, with 
Lipschitz-constant, say, $K_{\varepsilon}$.
Let $M$ be such that $\sup_{n}|u_{n}|\leq M$ and denote by $A_{d}$ the volume of the unit ball
in $\mathbb{R}^{d}$. We thus have, by the Markov inequality,
\begin{eqnarray*}
& & P\left(|\partial_i Q\left(\tilde{X}_{1}(\mathbf{q}_{n})-u_{n}\right)-\partial_i Q\left(\tilde{X}_{1}(\mathbf{q})-u\right)|\geq \varepsilon{}
\right)\\
&\leq& P\left(|\tilde{X}_{1}(\mathbf{q}_{n})-u_{n}|\leq \varepsilon\right){}
+P\left(|\tilde{X}_{1}(\mathbf{q})-u|\leq \varepsilon\right)\\
&+& K_{\varepsilon}E\left[|\tilde{X}_{1}(\mathbf{q}_{n})-u_{n}-\tilde{X}_{1}(\mathbf{q})+u|\right]/\varepsilon\\
&\leq& A_{d}\varepsilon^{d} \sup_{|v|\leq M+\varepsilon}\sup_{k\in\mathbb{N}}p_{\mathbf{q}_{k}}(v)+
A_{d}\varepsilon^{d} \sup_{|v|\leq M+\varepsilon}p_{\mathbf{q}}(v)+
K_{\varepsilon}
{E\left[|\tilde{X}_{1}(\mathbf{q}_{n})-u_{n}-\tilde{X}_{1}(\mathbf{q})+u|\right]}/{\varepsilon}.
\end{eqnarray*}
Here the third term is smaller than $\varepsilon$ for $n$ large enough, by Lemma \ref{contu}. Then
all three terms can be made arbitrarily small by choosing $\varepsilon$ small enough and $n$
large enough,
the suprema being finite by \eqref{morso}. This shows convergence in probability for the first factor
in $\eqref{concon}$.
 
Now we turn to the second factor in \eqref{concon}.
By Proposition 1.5.4 of \cite{nualart}, convergence of
\begin{equation}\label{esz}
\gamma^{i,j}(\tilde{X}_{1}(\mathbf{q}_{n}))D\tilde{X}_{1}(\mathbf{q}_{n})\to 
\gamma^{i,j}(\tilde{X}_{1}(\mathbf{q}))D\tilde{X}_{1}(\mathbf{q})\quad\mbox{ in }\mathbb{D}^{1,2}
\end{equation}
implies the $L^{2}$-convergence (hence also convergence in probability) of
$$
\delta\left(\gamma^{i,j}(\tilde{X}_{1}(\mathbf{q}_{n}))D\tilde{X}_{1}(\mathbf{q}_{n})\right)\to{}
\delta\left(\gamma^{i,j}(\tilde{X}_{1}(\mathbf{q}))D\tilde{X}_{1}(\mathbf{q})\right).{}
$$
So it remains to establish \eqref{esz}.

Since $\gamma^{i,j}(\tilde{X}_{1}(\mathbf{q}_{n}))$ are uniformly bounded and
$D\tilde{X}_{1}(\mathbf{q}_{n})\to D\tilde{X}_{1}(\mathbf{q})$ in $L^{2}([0,1]\times \Omega)$,{}
we clearly have that $\gamma^{i,j}(\tilde{X}_{1}(\mathbf{q}_{n}))D\tilde{X}_{1}(\mathbf{q}_{n})\to 
\gamma^{i,j}(\tilde{X}_{1}(\mathbf{q}))D\tilde{X}_{1}(\mathbf{q})$ in $L^{2}([0,1]\times\Omega)${}
by Lemma \ref{contu} and by the fact that matrix inversion is a continuous operation. Let us now
have a closer look at the Malliavin derivative of 
$\gamma^{i,j}(\tilde{X}_{1}(\mathbf{q}_{n}))D\tilde{X}_{1}(\mathbf{q}_{n})$.

First, let us recall that $D^{2}\tilde{X}_{1}(\mathbf{q}_{n})\to D^{2}\tilde{X}_{1}(\mathbf{q})$ in
$L^{p}([0,1]^{2}\times\Omega)$ for all $p\geq 1$, by Lemma \ref{contu}. Also,
$\gamma^{i,j}(\tilde{X}_{1}(\mathbf{q}_{n}))\to \gamma^{i,j}(\tilde{X}_{1}(\mathbf{q}))$ in $L^{p}$ and
$D\tilde{X}_{1}(\mathbf{q}_{n})\to D\tilde{X}_{1}(\mathbf{q})$ in $L^{p}$.
It remains to establish
\begin{equation}\label{itremains}
D\gamma^{i,j}(\tilde{X}_{1}(\mathbf{q}_{n}))\to D\gamma^{i,j}(\tilde{X}_{1}(\mathbf{q}))	
\end{equation}
in $L^{p}$. 

Denote by $G:\mathcal{V}\to\mathbb{R}^{d\times d}$ the operation of matrix inversion and
by $G'$ its derivative. We find that, on the set of positive definite matrices bounded
away from $0$, $G'$ is bounded. Now \eqref{itremains} follows from Lemma \ref{contu} and from
our previous observations. 
As these arguments work for all $i,j$, we indeed get $p_{\mathbf{q}_{n}}(u_{n})\to{}
p_{\mathbf{q}}(u)$.
\end{proof}

Proving the positivity of densities is an evergreen topic in Malliavin
calculus.
%, see e.g.\ \cite{arturo1,arturo2,bally,viens,delyon,crisan}
%for some more recent developments. 
We will rely on the deep study \cite{bally-caramellino} in the next lemma.

\begin{lemma}\label{lower}
We have $p_{\mathbf{q}}(u)>0$ for every $\mathbf{q}$, $u$.	
\end{lemma}
\begin{proof}
Fix $\mathbf{q}$, $u$. We will choose $0<\eta<1/2$ later.
We will apply Theorem 3.3 of \cite{bally-caramellino} with the choice $y=u$, $r=1$, $T=1$, $\delta=\eta$,
\begin{eqnarray*}
& & F=\tilde{X}_{1}(\mathbf{q}),\ F_{1-\eta}=\tilde{X}_{1-\eta}(\mathbf{q}),\\
& & G_{\eta}=\mathbf{v}_{1-\eta}\sqrt{I-\mathbf{r}_{1-\eta}\mathbf{r}_{1-\eta}^{*}}(W_{1}-W_{1-\eta}),\\
& & R_{\eta}=R^{1}_{\eta}+R^{2}_{\eta}=
\int_{1-\eta}^{1}(\mathbf{v}_{s}\sqrt{I-\mathbf{r}_{s}\mathbf{r}_{s}^{*}}
-\mathbf{v}_{1-\eta}\sqrt{I-\mathbf{r}_{1-\eta}\mathbf{r}_{1-\eta}^{*}})\, dW_{s}\\
&+& \int_{1-\eta}^{1}
\zeta(\tilde{X}_{s}(\mathbf{q})+\mathbf{z}_{s},\mathbf{v}_{s})\, ds. 
\end{eqnarray*}
We are now checking the conditions of that theorem.

First, $F_{1-\eta}$ is $\mathcal{F}_{1-\eta}$-measurable. Second, 
as all coordinates of $\tilde{X}_{1}(\mathbf{q})$ are in $\mathbb{D}^{2,\infty}$ 
by Lemma \ref{conditions}, so are
those of $R_{\delta}$. Third,
since $s\to (\mathbf{r}_{s},\mathbf{v}_{s})$
is continuous and $\mathbf{r}_{s}\in\mathbf{C}^{1+}$, $\mathbf{v}_{s}\in\mathbf{C}^{+}$, 
there is $\varepsilon>0${}
such that $\mathbf{v}_{s}\sqrt{I-\mathbf{r}_{s}\mathbf{r}_{s}^{*}}\geq \varepsilon I$ for all $s$. 
It follows that
$$
C_{\eta}:=\int_{1-\eta}^{1} \mathbf{v}^{*}_{s}[I-\mathbf{r}_{s}\mathbf{r}_{s}^{*}]\mathbf{v}_{s}\, ds\geq 
\eta\varepsilon^{2}I.
$$
Clearly, $\mathrm{det}(C_{\eta})\neq 0$. 

In order to apply Theorem 3.3 of \cite{bally-caramellino}, it remains to check that the event
$$
\tilde{\Gamma}_{\eta,1}=
\{|F_{1-\eta}-y|\leq 1/2\}\cap \{ ||C_{\eta}^{-1/2}R_{\eta}||_{\eta,2,q}\leq a\mathrm{e}^{-1}\}
$$
has positive probability for a suitable $\eta$. Here $q$ and $a$ are explicit
constants whose precise form can be found in \cite{bally-caramellino}. For $U\in\mathbb{R}^{d}$ with
all coordinates in $\mathbb{D}^{2,q}$ the norm $||U||_{\delta,2,q}$ 
is defined as the random quantity
\begin{equation}\label{normacska}
\left({E_{T-\eta}[|U|^{q}]+E_{T-\eta}\left[\left(\int_{T-\eta}^{T}
|D_{s}(U)|^{2}\, ds\right)^{q/2}\right]+
E_{T-\eta}\left[\left(\int_{T-\eta}^{T}\int_{T-\eta}^{T}
|D^{2}_{s_{1},s_{2}}(U)|^{2}\, ds_{1}\, ds_{2}\right)^{q/2}\right]}\right)^{1/q},
\end{equation}
with $E_{T-\eta}$ denoting conditional expectation with respect to $\mathcal{F}_{T-\eta}$.
As we have already seen, $$C_{\eta}^{-1/2}\leq \frac{1}{\varepsilon\sqrt{\eta}}I,$$
so it suffices to show that
$$
\hat{\Gamma}_{\eta,1}:=\{|F_{1-\eta}-u|\leq 1/2\}\cap \{ ||R_{\eta}||_{\eta,2,q}\leq 
a\mathrm{e}^{-1}{\varepsilon\sqrt{\eta}}\}
$$
has positive probability.

By an easy extension of the support theorem for
diffusions, see \cite{bubble}, the process $\tilde{X}(\mathbf{q})$ has
full support on the space of continuous functions starting from $x$, so we clearly have that
$$
P(A_{\eta})>0\mbox{ for }A_{\eta}:=\{|F_{1-\eta}-u|\leq 1/2\},{}
$$
for each $0<\eta<1/2$.
A standard argument (like Lemma 19 of \cite{bally}) shows that 
$$
\|R_{\eta}\|_{\eta,2,q}\leq (1+|F_{1-\eta}|)\eta
$$
almost surely. But then on $A_{\eta}$ we have $\|R_{\eta}\|_{\delta,2,q}\leq (1+|u|+1/2)\eta$. 
Clearly, this is 
smaller than $a\mathrm{e}^{-1}{\varepsilon\sqrt{\eta}}$ for $\eta$ small enough.
We conclude
that the set $\hat{\Gamma}_{\eta,q}$ contains $A_{\eta}$ for $\eta$ small enough,
consequently it has positive probability. Now Theorem
3.3 of \cite{bally-caramellino} implies that $p_{\mathbf{q}}(\cdot)\geq c$
Lebesgue-a.s.\ with some $c>0$ in a neighbourhood of $u$. As $p_{\mathbf{q}}(u)$ is 
continuous in $u$, $p_{\mathbf{q}}(u)\geq c$,
showing our lemma.
\end{proof}

\begin{corollary}\label{kaki}
There exist constants $\tilde{c}_n>0$, $n\in\mathbb{N}$ such that 
for each $A\in\mathcal{B}(\mathbb{R})$ with $A\subset [-1,1]$ and for all 
$(\mathbf{v},\mathbf{z},\mathbf{r})\in \mathcal{Y}_n$, $x\in\mathcal{X}_n$,
$$
P(\tilde{X}_1(\mathbf{v},\mathbf{z},\mathbf{r},x)\in A)\geq \tilde{c}_n\mathrm{Leb}(A).
$$
\end{corollary}
\begin{proof} 
Compactness of $[-1,1]\times\mathcal{Y}_{n}\times\mathcal{X}_{n}$, Lemmas
\ref{compact} and \ref{lower} imply that 
$$
\inf_{(u,\mathbf{q})\in [-1,1]\times \mathcal{Y}_n\times \mathcal{X}_n}
p_{\mathbf{q}}(u)>0.
$$
\end{proof}

Define $\hat{X}_t(\mathbf{v},\mathbf{z},\mathbf{r},x):=\tilde{X}_t(\mathbf{v},\mathbf{z},\mathbf{r},x)
+\mathbf{z}_{t}$, $t\in [0,1]$. This process satisfies the integral equation
\begin{equation*}
\hat{X}_t(\mathbf{v},\mathbf{z},\mathbf{r},x)=x+
\int_{0}^{t}\zeta(\hat{X}_s,\mathbf{v}_{s})\, ds+\mathbf{z}_{t}+
\int_{0}^{t}\mathbf{v}_{s}\sqrt{I-\mathbf{r}_{s}\mathbf{r}_{s}^{*}}\, dW_{s},\ t\in [0,1]
\end{equation*}
hence it will serve as the ``parametric version'' of \eqref{startrek}.

\begin{corollary}\label{compact1}  
There exist constants $\hat{c}_n>0$, $n\in\mathbb{N}$ such that 
for each $A\in\mathcal{B}(\mathbb{R})$ with $A\subset [-1,1]$ and for all 
$(\mathbf{v},\mathbf{z},\mathbf{r})\in \mathcal{Y}_n$, $x\in\mathcal{X}_n$,
$$
P(\hat{X}_1(\mathbf{v},\mathbf{z},\mathbf{r},x)\in A)\geq \hat{c}_n\mathrm{Leb}(A).
$$
\end{corollary}
\begin{proof} Note that
$(\mathbf{v},\mathbf{z},\mathbf{r})\to ||\mathbf{z}||_{\mathbf{C}_{d}}$ is 
bounded on each $\mathcal{Y}_{n}$. Hence there is $N\geq n$ such that whenever $x\in\mathcal{X}_{n}$
one has $x+\mathbf{z}_{1}\in \mathcal{X}_{N}$ for all $\mathbf{q}\in\mathcal{Y}_{n}$.
Now Corollary
\ref{kaki} readily implies the statement.
\end{proof}

%Let $\mathbf{C}[0,1]$ denote the set of $\mathbb{R}^{d}$-valued continuous functions
%on $[0,1]$.

\begin{lemma}\label{megoldja} There exists a measurable 
mapping $\Xi:\Omega\times\cup_{n}\mathcal{Y}_{n}\times\mathbb{R}\to{}
\mathbf{C}_{d}$ such that it satisfies for all 
$\mathbf{q}\in\cup_{n}\mathcal{Y}_{n}\times\mathbb{R}$ the equation
\begin{equation*}
{\Xi}_{t}(\mathbf{q})=x+\int_{0}^{t}\zeta(\Xi_{s}(\mathbf{q}),\mathbf{v}_{s})\, ds+
\mathbf{z}_{t}+\int_{0}^{t}\mathbf{v}_{s}\sqrt{I-\mathbf{r}_{s}\mathbf{r}_{s}^{*}}\, dW_{t},
\ t\in [0,1].
\end{equation*}
For almost all $\omega$, $\Xi(\omega,\cdot,\cdot)$ is continuous.
Furthermore, $\Xi_{t}(\cdot,Y_{k},L_{k})$, $t\in [0,1]$ is a version of 
$L_{k+t}$, $t\in [0,1]$. From now on we always take this version of $L$.
\end{lemma}
\begin{proof} 
Let us take an increasing sequence of sets $B_{n}\subset\mathcal{Y}_{n}\times\mathcal{X}_{n}$, $n\in\mathbb{N}$
which are countable and dense in $\mathcal{Y}_{n}\times\mathcal{X}_{n}$.
By Lemma \ref{contu}, there is a common $P$-null set 
$N\in\mathcal{F}$ such that for $\omega\in\Omega\setminus N$
the mapping $\mathbf{q}\to (\hat{X}_{u}(\mathbf{q})(\omega))_{u\in [0,1]}\in\mathbf{C}_{d}$ is uniformly continuous on 
$B_{n}$ for each $n$ hence it has a continuous extension to $\mathcal{Y}_{n}\times\mathcal{X}_{n}$ 
which coincides with the respective extensions on $\mathcal{Y}_{l}\times\mathcal{X}_{l}$ 
for $l\leq n$. Hence we eventually get a function 
$\Xi:(\Omega\setminus N)\times\cup_{n}\mathcal{Y}_{n}\times\mathbb{R}\to{}
\mathbf{C}_{d}$ that is measurable in its first variable and jointly continuous
in its second and third, hence jointly measurable in all three variables. (We set $\Xi:=0$ on $N$.)

For any $\mathcal{G}_{\infty}\vee\mathcal{F}_{k}$-measurable 
step function  $\mathbf{Q}:\Omega\to\cup_{n}\mathcal{Y}_{n}\times\mathbb{R}$
with $\mathbf{Q}=(\mathbf{V},\mathbf{Z},\mathbf{R},X)$ it clearly holds that 
\begin{equation*}
{\Xi}_{t}(\mathbf{Q})=X+\int_{0}^{t}
\zeta(\Xi_{s}(\mathbf{Q})),\mathbf{v}_{s})\, ds+\mathbf{Z}_{t}+\int_{0}^{t}
\mathbf{v}_{s}\sqrt{I-\mathbf{R}_{s}\mathbf{R}_{s}^{*}}\, dW_{t},
\ t\in [0,1],
\end{equation*}
and then this extends by continuity to all $\cup_{n}\mathcal{Y}_{n}\times\mathbb{R}$-valued 
$\mathcal{G}_{\infty}\vee\mathcal{F}_{k}$-measurable random variables $\mathbf{Q}$, by continuity.
In particular, it holds for $\mathbf{Q}:=(Y_{k},L_{k})$, which proves the second statement.
\end{proof}

Let us define the parametrized kernel $Q$ as follows: for each 
$(x,y)\in\mathbb{R}\times\cup_{n}\mathcal{Y}_{n}$ and
for all continuous and bounded $\phi:\mathbb{R}^{d}\to\mathbb{R}^{d}$ we let
$$
\int_{\mathbb{R}^{d}}\phi(z)\, Q(x,y,dz):=E[\phi(\Xi_{1}(y,x))].
$$
This clearly defines a probability for all $(x,y)$, and for a fixed 
$\phi$ it is measurable in $(x,y)$ by Lemma
\ref{megoldja}.
Now we can recursively generate
$$
X_0:=L_{0},
\ X_{t+1}:={\Xi}_{t+1}(Y_{t},X_{t}), t\in\mathbb{N}\setminus\{0\}
$$
and see that $X$ is a Markov chain in random environment with kernel $Q$
which satisfies $X_{t}=L_{t}$, $t\in\mathbb{N}$.
Notice that \eqref{tighti} holds by Lemma \ref{lyapunov} above.

Let $\mu,\nu$ be probabilities on $\mathcal{B}(\mathbb{R}^{d}\times\mathcal{W}^{d\times d}
\times\mathcal{W}^{d\times m})$.
Let $\mathcal{C}(\mu,\nu)$ denote the set of probabilities $\pi$
on $\mathcal{B}(\mathbb{R}^{d}\times\mathcal{W}^{d\times d}
\times\mathcal{W}^{d\times m}\times \mathbb{R}\times\mathcal{W}^{d\times d}
\times\mathcal{W}^{d\times m})$ such 
that their respective marginals are $\mu,\nu$. Define
\begin{eqnarray*}
& & \mathbf{w}(\mu,\nu):=\\
& & \inf_{\zeta\in\mathcal{C}(\mu,\nu)}
\int_{(\mathbb{R}^{d}\times\mathcal{W}^{d\times d}
\times\mathcal{W}^{d\times m})^{2}} 
\left(\left[1 \wedge |x_{1}-x_{2}|\right]+\mathbf{d}_{d\times d}(v_{1},v_{2})
+\mathbf{d}_{d\times m}(w_{1},w_{2})\right)
\pi(dx_{1},dv_{1},dw_{1},dx_{2},dv_{2},dw_{2}).
\end{eqnarray*}
This bounded Wasserstein distance metrizes weak convergence 
of probabilities on $\mathcal{B}(\mathbb{R}^{d}\times\mathcal{W}^{d\times d}
\times\mathcal{W}^{d\times m})$ and 
satisfies $\mathbf{w}(\mu,\nu)\leq C||\mu-\nu||_{TV}$ for some $C>0$,
see Theorem 6.15 of \cite{villani}.

\begin{proof}[Proof of Theorem \ref{stability}.] The letter $C$ refers to various constants in this proof.
Invoking Theorem \ref{maine1}, we can establish the existence of $\mu_{\sharp}$ such that 
$$
\mathcal{L}(L_l,\mathbf{V}_{l},\mathbf{R}_{l})\to \mu_{\sharp},\ l\to\infty,\ l\in\mathbb{N}
$$
holds in $||\cdot||_{TV}$. Working on a finer time grid,
we similarly obtain that, for each $k\in\mathbb{N}$, the sequence of laws 
$\mathcal{L}(L_{l/2^{k}},\mathbf{V}_{l/2^{k}},\mathbf{R}_{l/2^{k}})$, $l\in\mathbb{N}$ 
converge in $||\cdot||_{TV}$ as $l\to\infty$ and all these limits necessarily equal $\mu_{\sharp}$.

Assumption \ref{thrice} implies Lipschitz-continuity of $\zeta$ in its first variable
and local Lipschitz-continuity with linearly growing Lipschitz-continuity in its
second variable. In particular, $|\zeta(x,v)|\leq C(1+|x|+|v|^{2})$, hence for $0<h\leq 1$, 
\begin{eqnarray}
& & E[|L_{t+h}-L_t|^2]\nonumber \\
\nonumber &\leq &  3 
E\left[\left(\int_t^{t+h} \zeta\left(L_s,V_{s}\right)\, ds\right)^2\right] 
+ 3E\left[\left(\int_t^{t+h} V_s\rho_{s}\, dB_s\right)^2\right] +
3E\left[\left(\int_t^{t+h} 
\sqrt{I-\rho_{s}\rho_{s}^{*}}V_s\, dW_s\right)^2\right]
\\ \nonumber &\leq&  
\int_t^{t+h} C\left[E[|L_s|^2]+E[|V_0|^4]+1\right]\, ds
+3\int_t^{t+h}E[|V_0|^2]\, ds+3\int_t^{t+h}E[|V_0|^2]\, ds\\
&\leq& hC[\tilde{L} +E[|V_0|^4]+1+E[|V_0|^2]]\leq Ch,\label{hru}
\end{eqnarray}
by Assumption \ref{thrice} and Lemma \ref{lyapunov}. It is only at this point
that we need $E[|V_{0}|^{4}]<\infty$.

For each $t\in\mathbb{R}_+$ and $k\in\mathbb{N}$, let $l(k,t)$ denote the integer satisfying 
$l(k,t)/2^k\leq t<[l(k,t)+1]/2^k$. Notice that, for $k$ fixed, $l(k,t)\to\infty$ as $t\to\infty$.
We estimate, using \eqref{hru},
\begin{eqnarray*}
& & \mathbf{w}(\mathcal{L}(L_t,\mathbf{V}_{t},\mathbf{R}_{t}),\mu_{\sharp})\\
&\leq& \mathbf{w}(\mathcal{L}(L_t,\mathbf{V}_{t},\mathbf{R}_{t}),
\mathcal{L}(L_{l(k,t)/2^k},\mathbf{V}_{l(k,t)/2^k},\mathbf{R}_{l(k,t)/2^k}))
+ \mathbf{w}(\mathcal{L}(L_{l(k,t)/2^k},\mathbf{V}_{l(k,t)/2^k},\mathbf{R}_{l(k,t)/2^k}),\mu_{\sharp})\\
&\leq& E|L_t-L_{l(k,t)/2^k}| + E[\mathbf{d}_{d\times d}(\mathbf{V}_{t},\mathbf{V}_{l(k,t)/2^k})]\\
&+& E[\mathbf{d}_{d\times m}(\mathbf{R}_{t},\mathbf{R}_{l(k,t)/2^k}] +
C||\mathcal{L}(L_{l(k,t)/2^k},\mathbf{V}_{l(k,t)/2^k},\mathbf{R}_{l(k,t)/2^k})-\mu_{\sharp}||_{TV}\\
&\leq& \sqrt{{C}/{2^{k}}}+
\sup_{t\in\mathbb{R}}\{E[\mathbf{d}_{d\times d}(\mathbf{V}_{t},\mathbf{V}_{l(k,t)/2^k})]+ 
E[\mathbf{d}_{d\times m}(\mathbf{R}_{t},\mathbf{R}_{l(k,t)/2^k}]\} \\
&+& C||\mathcal{L}(L_{l(k,t)/2^k})-\mu_{\sharp}||_{TV}.
\end{eqnarray*}
Noting Lemma \ref{ucont} and Theorem \ref{maine1}, the latter expression can be made arbitrarily small by 
first choosing $k$ large enough and then choosing $t$ large enough.

Now we turn to proving stationarity. Theorem \ref{maine1} implies that, if $\mathcal{L}(L_{0},\mathbf{V}_{0},\mathbf{R}_{0})=\mu_{\sharp}$
then
\begin{equation}\label{egyenlo}
\mathcal{L}(L_{t},\mathbf{V}_{t},\mathbf{R}_{t})=\mu_{\sharp}
\end{equation}
holds for all dyadic rationals $t\geq 0$. For an arbitrary 
$t\in\mathbb{R}$, take dyadic rationals $t_{n}\to t$, $n\to\infty${}
and estimate
\begin{eqnarray*} & &
\mathbf{w}(\mathcal{L}(L_t,\mathbf{V}_{t},\mathbf{R}_{t}),
\mathcal{L}(L_{t_{n}},\mathbf{V}_{t_{n}},\mathbf{R}_{t_{n}}))\\
&\leq& 
E|L_t-L_{t_{n}}| + E[\mathbf{d}_{d\times d}(\mathbf{V}_{t},\mathbf{V}_{t_{n}})]+
E[\mathbf{d}_{d\times m}(\mathbf{R}_{t},\mathbf{R}_{t_{n}})],
\end{eqnarray*}
which tends to $0$ as $n\to\infty$, by Lemma \ref{ucont} and by \eqref{hru}. 
Hence \eqref{egyenlo} holds for all $t\in\mathbb{R}$.
\end{proof}

\begin{proof}[Proof of Theorem \ref{stability1}.]
Notice that, in the proof of Theorem \ref{stability}, we used Assumption \ref{dissipi}
only in Lemma \ref{lyapunov}.
Under our current assumptions, we will verify
$$
\sup_{t\geq 0}E[e^{\kappa|L_{t}|}]<\infty
$$
for some $\kappa>0$. This trivially entails
$$
\tilde{L}:=\sup_{t\geq 0}E[|L_{t}|^{2}]<\infty,
$$
and the rest of the proof follows verbatim that of Theorem \ref{stability}.

We will use the Lyapunov-function $g(x):=\exp\left(\kappa\sqrt{1+|x|^{2}}\right)$, $x\in\mathbb{R}^{d}$,
where $0<\kappa\leq\kappa_{0}$ will be chosen later.
Note that 
$$
\partial_{i}g(x)=\exp\left(\kappa\sqrt{1+|x|^{2}}\right)\frac{\kappa x^{i}}{\sqrt{1+|x|^{2}}},\ i=1,\ldots,d,{}
$$
and $|\partial_{ij}g(x)|\leq C_{0}\kappa g(x)$ for all $x$, with some constant $C_{0}>0$,
for all $1\leq i,j\leq d$.

Fix $k\in\mathbb{N}$.
Define the stopping times $\tau_l:=\inf\{ t>k:|L_t|>l\}$ for $l\in\mathbb{N}$.
Apply It\^{o}'s lemma to obtain
\begin{eqnarray*}
\mathrm{e}^{\alpha(t\wedge \tau_{l}-k)}\mathrm{e}^{\kappa \sqrt{1+|L_{t\wedge\tau_{l}}|^{2}}} &\leq& 
\mathrm{e}^{\kappa \sqrt{1+|L_{k}|^{2}}}+\int_{k}^{t\wedge\tau_{l}}\mathrm{e}^{\alpha(s-k)}
\kappa\frac{\mathrm{e}^{\kappa\sqrt{1+|L_{s}|^{2}}}}{\sqrt{1+|L_{s}|^{2}}}
\langle L_{s},\zeta(L_{s},V_{s})\rangle\, ds\\
&+& \int_{k}^{t\wedge\tau_{l}}\mathrm{e}^{\alpha(s-k)}
\kappa\frac{\mathrm{e}^{\kappa\sqrt{1+|L_{s}|^{2}}}}{\sqrt{1+|L_{s}|^{2}}}
L_{s}^{*}V_{s}\, d\overline{W}_{s} +\int_{k}^{t\wedge\tau_{l}}C_{1}\kappa	
\mathrm{e}^{\alpha(s-k)}\mathrm{e}^{\kappa\sqrt{1+|L_{s}|^{2}}}|V_{s}|^{2}\, ds\\
&+& \int_{k}^{t\wedge\tau_{l}}\alpha \mathrm{e}^{\alpha(s-k)}
\mathrm{e}^{\kappa\sqrt{1+|L_{s}|^{2}}}\, ds,\ t\geq k,
\end{eqnarray*}
for some $C_{1}>0$.
Taking expectations, using the martingale property of stochastic integrals
and \eqref{madi}, we arrive at
\begin{eqnarray*}
E\left[e^{\alpha(t\wedge\tau_{l}-k)}\mathrm{e}^{\kappa\sqrt{1+|L_{t\wedge\tau_{l}}|^{2}}}\right] &\leq& 
E\left[\mathrm{e}^{\kappa\sqrt{1+|L_{k}|^{2}}}\right]+E\left[\int_{k}^{t\wedge\tau_{l}}
\kappa\mathrm{e}^{\alpha(s-k)}\frac{\mathrm{e}^{\kappa\sqrt{1+|L_{s}|^{2}}}}{\sqrt{1+|L_{s}|^{2}}}
(-\alpha|L_{s}|^{1+\gamma}+\beta(1+|V_{s}|^{\xi}))\, ds\right]\\
&+& E\left[\int_{k}^{t\wedge\tau_{l}}C_{1}\kappa	
\mathrm{e}^{\alpha(s-k)} \mathrm{e}^{\kappa\sqrt{1+|L_{s}|^{2}}}|V_{s}|^{2}\, ds\right]
+ E\left[\int_{k}^{t\wedge\tau_{l}}\alpha \mathrm{e}^{\alpha(s-k)}
\mathrm{e}^{\kappa\sqrt{1+|L_{s}|^{2}}}\, ds\right].
\end{eqnarray*}
Set $C_{2}:=C_{1}+\beta$. Let us notice that, on the event $$
A:=\left\{|L_{s}|\geq \max\left\{1,
\left(\frac{2\sqrt{2}}{\kappa}\right)^{1/\gamma}+
\left(\frac{2\sqrt{2} C_{2}}{\alpha}\right)^{1/\gamma}(1+|V_{s}|^{\xi})^{1/\gamma}\right\} \right\}
$$ 
we have
\begin{eqnarray*}
\frac{-\alpha\kappa}{\sqrt{1+|L_{s}|^{2}}}
|L_{s}|^{1+\gamma}+C_{2}\kappa(1+|V_{s}|^{\xi})+\alpha
&\leq& 0
\end{eqnarray*}
On the complement of $A$, 
$$
\exp\left(\kappa
\sqrt{1+|L_{s}|^{2}}\right)\leq \exp\left({}
\kappa+\kappa\left(1+\left(\frac{2\sqrt{2}}{\kappa}\right)^{1/\gamma}+
\left(\frac{2\sqrt{2} C_{2}}{\alpha}\right)^{1/\gamma}(1+|V_{s}|^{\xi})^{1/\gamma}\right)\right).
$$
Tending $l\to\infty$ and applying Fatou's lemma, our estimate takes the form
\begin{eqnarray*}
& & E\left[e^{\alpha(t-k)}\mathrm{e}^{\kappa\sqrt{1+|L_{t}|^{2}}}\right] \leq  
E\left[\mathrm{e}^{\kappa \sqrt{1+|L_{k}|^{2}}}\right]\\
&+&
C_{3}(\kappa)\int_{k}^{k+1}E\left[(1+|V_{s}|^{\xi}) 
\exp\left(\kappa\left(\frac{2\sqrt{2} C_{2}}{\alpha}\right)^{1/\gamma}(1+|V_{s}|^{\xi})^{1/\gamma}\right){}
\right]\, ds\\
&\leq& 
E\left[\mathrm{e}^{\kappa \sqrt{1+|L_{k}|^{2}}}\right]+
C_{4}(\kappa)E\left[(1+|V_{0}|^{\xi}) 
\exp\left(\kappa\left(\frac{2\sqrt{2} C_{2}}{\alpha}\right)^{1/\gamma}2^{{1}/{\gamma}}|V_{0}|^{\xi/\gamma}\right)\right].
\end{eqnarray*}
for all $k\leq t\leq k+1$, with constants $C_{3}(\kappa),C_{4}(\kappa)$. Choosing $\kappa$ such that  
$\left(\frac{2\sqrt{2} C_{2}}{\alpha}\right)^{1/\gamma}2^{{1}/{\gamma}}\kappa<\kappa_{0}$,
the second integral is finite, by \eqref{madi}. 
Now we can easily conclude,
as in Lemma \ref{lyapunov} above.
\end{proof}

\textbf{Data availability statement:} there is no associated data to this manuscript.

\begin{thebibliography}{99}

\bibitem{bally} V. Bally.
\newblock Lower bounds for the density of locally elliptic
It\^{o} processes.
\newblock \emph{Annals of Probability}, 34:2406--2440, 2006.

\bibitem{bally-caramellino} V. Bally and L. Caramellino.
\newblock Positivity and lower bounds for the density of Wiener functionals.
\newblock \emph{Potential Analysis}, 39:141--168, 2013.

\bibitem{bally-notes} V. Bally and L. Caramellino.
\newblock Integration by parts formulas, Malliavin calculus and
regularity of probability laws.
\newblock \emph{In: V. Bally, L. Caramellino and R. Cont: Stochastic integration
by parts and functional It\^{o} calculus}, 1--114, Birkh\"auser, 2016.

\bibitem{friz} Ch. Bayer, P. Friz and J. Gatheral.
\newblock Pricing under rough volatility.
\newblock\emph{Quantitative Finance}, 16:887--904, 2016.

\bibitem{bmp} A. Benveniste, M. M\'etivier and P. Priouret. 
\newblock\emph{Adaptive algorithms and stochastic approximations.}
Springer, 1990.

\bibitem{commodity1} H. Bessembinder, J. F. Coughenor, P. J. Seguin and M. M. Smoller.
\newblock Mean-reversion in equilibrium asset prices: evidence from the futures term
structure.
\newblock \emph{Journal of Finance}, 50:361--375, 1995.

%\bibitem{bm} R. Bhattacharya and M. Majumdar.
%\newblock On a theorem of Dubins and Freedman.
%\newblock \emph{J. Theor. Probab.}, 12:1067--1087, 1999.

\bibitem{bwbook} R. Bhattacharya and E. Waymire.
\newblock \emph{Stochastic Processes with Applications}.
\newblock Wiley \& Sons, New York, 1990.

%\bibitem{bw} R. Bhattacharya and E. Waymire.
%\newblock An approach to the existence
%of unique invariant probabilities for Markov processes.
%\newblock In: \emph{Limit theorems
%in probability and statistics, J\'anos Bolyai Math. Soc., I}, 181--200, 2002.

\bibitem{commodity2} J. Casassus and P. Collin-Dufresne.
\newblock Stochastic convenience yield implied from commodity futures and
interest rates. \newblock \emph{Journal of Finance}, 60:2283--2331, 2005.

\bibitem{cr}
F. Comte and \'E. Renault.
\newblock Long memory in continuous-time stochastic volatility models.
\newblock {\em Math. Finance}, 8:291--323, 1998.

%\bibitem{crisan} D. Crisan and E. McMurray.
%\newblock Smoothing properties of McKean–Vlasov SDEs.
%\newblock\emph{Probability Theory and Related Fields}, 171:97--148, 2018.

%\bibitem{delyon} B. Delyon and F. Portier.
%\newblock Integral approximation by kernel smoothing.
%\newblock \emph{Bernoulli}, 22:2177--2208, 2016.

\bibitem{diaconis-freedman} P. Diaconis and D. Freedman.
\newblock Iterated random functions.
\newblock \emph{SIAM Review}, 41:45--76, 1999.

\bibitem{lukasz}
G.B. Di Masi and L. Stettner.
\newblock Risk sensitive control of discrete time Markov processes
with infinite horizon.
\newblock \emph{SIAM J. Control Optimiz.}, 38:61--78, 2000.

\bibitem{down} D. Down, S. P. Meyn and R. L. Tweedie.
\newblock Exponential and uniform ergodicity of Markov processes.
\newblock \emph{Ann. Probab.}, 23:1671--1691, 1995.

\bibitem{asymp} H. F\"ollmer and W. Schachermayer. 
\newblock Asymptotic Arbitrage and large deviations. 
\newblock \emph{Math. Financ. Econ.}, 1:213--249, 2008.

\bibitem{fukasawa} M. Fukasawa.
\newblock Short-time at-the-money skew and rough fractional volatility.
\newblock \emph{Quantitative Finance}, 17:189--198, 2017.

\bibitem{gjr}
J. Gatheral, T. Jaisson and M. Rosenbaum. \newblock Volatility is rough.
\newblock\emph{Quantitative Finance}, 18:933--949, 2018. 

\bibitem{balazs}
B. Gerencs\'er and M. R\'asonyi.{}
\newblock On the ergodicity of certain {M}arkov chains in random environments.
\newblock\emph{Submitted}, 2020. arXiv:1807.03568v3

\bibitem{timid} P. Guasoni, L. Nagy, M. R\'asonyi.{}
\newblock Young, timid and risk seekers.
\newblock \emph{To appear in Mathematical Finance}, 2021. SSRN:3763151

\bibitem{grs} P. Guasoni, Zs. Nika and M. R\'asonyi.
Trading fractional Brownian motion.  
\emph{SIAM J. Financial Mathematics}, vol. 10, 769--789, 2019. 


\bibitem{bubble} P. Guasoni and M. R\'asonyi.{}
\newblock Fragility of arbitrage and bubbles
in local martingale diffusion models.
\newblock\emph{Finance Stoch.}, 19:215--231, 2015.

\bibitem{jacquier}
H. Guennoun, A. Jacquier, P. Roome and F. Shi.
\newblock Asymptotic behavior of the fractional Heston model.
\newblock \emph{SIAM J. Finan. Math.}, 9:1017--1045, 2018.

\bibitem{hairer-e} M. Hairer. 
\newblock Ergodicity of stochastic differential equations driven by fractional Brownian motion.
\newblock \emph{Annals of
Probability}, 703--758, 2005.

\bibitem{hairer}
M. Hairer.
\newblock \emph{Convergence of Markov processes.}
\newblock Lecture notes, 2016.\newblock http://hairer.org

\bibitem{hairer-mattingly}
M. Hairer and J. Mattingly.
\newblock Yet another look at Harris' ergodic theorem for Markov chains.
\newblock \emph{In: Seminar on stochastic analysis, random fields and applications VI (eds.
R. Dalang, M. Dozzi and F. Russo F.)}, Progress in Probability, vol. 63, 109--117, 2011.


\bibitem{hernandez-lerma} 
O. Hernandez-Lerma and J. B. Lasserre.
\newblock\emph{Discrete-time Markov control processes}.
Springer, 1996.



\bibitem{kha} R. Z. Khasminskii.
\newblock \emph{Stochastic stability of differential equations.} 2nd edition. 
\newblock Springer, 2012.

\bibitem{kifer1} Y. Kifer. \newblock Perron-Frobenius theorem, large deviations, and random perturbations 
in random environments. \newblock\emph{Math. Zeitschrift}, 222:677--698, 1996.

\bibitem{kifer2} Y. Kifer. \newblock Limit theorems for random transformations and
processes in random environments. \newblock\emph{Trans. American Math. Soc.}, 350:1481--1518, 1998.

%\bibitem{arturo1} A. Kohatsu-Higa.
%Lower bounds for densities of uniformly elliptic random variables on Wiener space.
%\newblock \emph{Probability Theory and Related Fields}, 126:421--457, 2003.


%\bibitem{arturo2} A. Kohatsu-Higa and A. Makhlouf.
%Estimates for the density of functionals of SDEs with irregular drift.
%\newblock \emph{Stoch. Proc. Appl.}, 123:1716--1728, 2013.


\bibitem{krylov} N. V. Krylov.
\newblock\emph{Controlled diffusion processes.}
\newblock Springer, 1980.

\bibitem{attila} A. Lovas and M. R\'asonyi.{}
\newblock Markov chains in random environment with applications
in queuing theory and machine learning.
\newblock \emph{To appear in Stochastic Processes and their Applications}, 2021. arXiv:1911.04377v2

\bibitem{malliavin-thalmaier} P. Malliavin and A. Thalmaier.
\newblock \emph{Malliavin calculus for financial applications.}
\newblock Springer, 2010.

\bibitem{mt} S. P. Meyn and R. L. Tweedie. 
\newblock \emph{Markov chains and stochastic stability.} 2nd edition.
\newblock Cambridge University Press, 2009.

\bibitem{mt2} S. P. Meyn and R. L. Tweedie. 
\newblock Stability of Markovian processes II. Continuous time
processes and sampled chains.
\newblock\emph{Adv. Appl. Prob.} 25:487--517, 1993.

\bibitem{mt3} S. P. Meyn and R. L. Tweedie. 
\newblock Stability of Markovian processes. III. Foster-Lyapunov criteria 
for continuous-time processes. \newblock\emph{Adv. Appl. Probab.},
25:518--548, 1993.

\bibitem{negmem} L. Nagy, M. R\'asonyi.{}
\newblock Optimal long-term investment in illiquid markets when prices have
negative memory. 
\newblock \emph{To appear in Electronic Communications in Probability}, 2021. arXiv:2005.07080


\bibitem{nualart} D. Nualart.
\newblock \emph{The Malliavin calculus and related topics.} 
\newblock 2nd edition, Springer, 2006.

\bibitem{pitera}
M. Pitera and L. Stettner.
\newblock Long run risk-sensitive portfolio with general factors.
\newblock \emph{Math. Method. Oper. Res.} 83:265--293, 2016.

\bibitem{propp-wilson} J. G. Propp and D. B. Wilson.
\newblock Exact sampling with coupled Markov chains and applications
to statistical mechanics.
\newblock \emph{Random Structures and Algorithms}, 9:223--252, 1996.

\bibitem{rt} G. O. Roberts and R. L. Tweedie. 
\newblock Exponential convergence of Langevin
distributions and their discrete approximations. 
\newblock \emph{Bernoulli}, 2:341--363, 1996.

\bibitem{sep} T. Sepp\"al\"ainen. 
\newblock Large deviations for Markov chains with random transitions. 
\newblock\emph{Ann. Probab.}, \newblock 22:713--748, 1994.

\bibitem{stenflo} \"O. Stenflo. Markov chains in random environments and
random iterated function systems. \emph{Trans. American Math. Soc.}, 353:3547--3562, 2001.

\bibitem{veretennikov}
A. Yu. Veretennikov.
\newblock Bounds for the mixing rate in the theory of stochastic equations.
\newblock \emph{Theory Probab. Appl.}, 32:273--281, 1987.

\bibitem{ver2}
A. Yu. Veretennikov. 
\newblock On polynomial mixing bounds for stochastic
differential equations.
\newblock \emph{Stoch. Processes Appl.}, 70:115--127, 1997.

\bibitem{villani} C. Villani.
\newblock \emph{Optimal transport. Old and new.}
\newblock Springer, 2009.

%\bibitem{viens} F. G. Viens.
%\newblock Stein's lemma, Malliavin calculus, and tail bounds, with application to polymer fluctuation exponent.
%\newblock \emph{Stochastic Processes and their Applications}, 119:3671-3698, 2009.


\bibitem{alfa} C. S. Withers. \newblock Central limit theorems for dependent
variables I. \newblock\emph{Z. Wahsch. verw. Gebiete}, 57:509--534, 1981.

\bibitem{website} https://sites.google.com/site/roughvol/home/risks-1

\end{thebibliography}
\end{document}